\documentclass[9pt]{amsart}
\usepackage{amssymb}
\usepackage[mathscr]{euscript}
\usepackage[all]{xy}
\usepackage{amsmath,epsfig}

\newtheorem{observation}{Observation}

\newtheorem{theorem}{Theorem}
\newtheorem{lemma}[theorem]{Lemma}

\newtheorem{corollary}[theorem]{Corollary}
\begin{document}

\markboth{Micah W. Chrisman}
{Twist Lattices for Virtual Knots}

\title{Twist Lattices and the Jones-Kauffman Polynomial for Long Virtual Knots}

\author{Micah W. Chrisman}

\maketitle

\begin{abstract}
In this paper, we investigate twist sequences for Kauffman finite-type invariants and Goussarov-Polyak-Viro finite-type invariants.  It is shown that one obtains a Kauffman or GPV type of degree $\le n$ if and only if an invariant is a polynomial of degree $\le n$ on every twist lattice of the right form.  The main result of this paper is an application of this technique to the coefficients of the Jones-Kauffman polynomial.  It is shown that the Kauffman finite-type invariants obtained from these coefficients are not GPV finite-type invariants of any degree by explicitly showing they can never be polynomials.  This generalizes a result of Kauffman \cite{virtkauff}, where it is known for degree $k=2$.
\end{abstract}

\section{Introduction}
\footnote{This is a preprint of an article submitted to the Journal of Knot Theory and Its Ramifications and later accepted for publication in March 2009.}Twist sequences have a long and distinguished history in the study of Vassiliev invariants.  These sequences were first studied by Dean \cite{MR1265449} and Trapp \cite{MR1291867}.  They have typically been used to show that an invariant is not of finite-type.  For example, this method has been used to show that the crossing number, the unknotting number, the coefficients of the Jones polynomial, and the number of knot group representations are not of finite-type (see \cite{MR1291867}, \cite{MR1265449} and \cite{MR1657727}).  In the classical case, an example of a type of twist sequence is given below:
\begin{equation} \label{realtwist}
\begin{array}{cccccc} \cdots & \scalebox{.2}{\psfig{figure=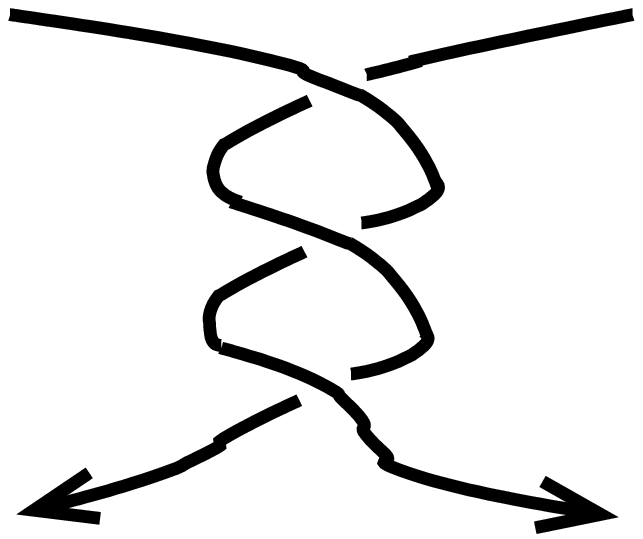}} & \scalebox{.2}{\psfig{figure=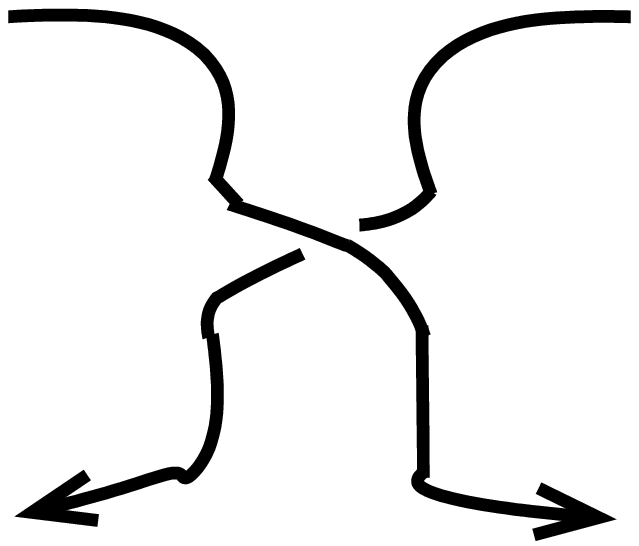}} & \scalebox{.2}{\psfig{figure=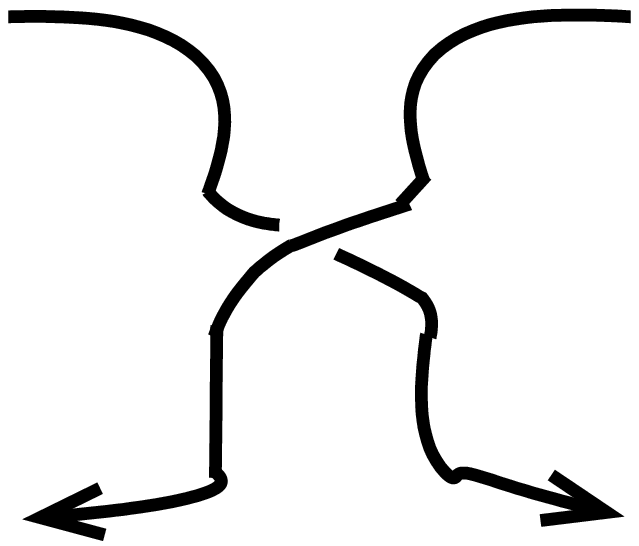}}& \scalebox{.2}{\psfig{figure=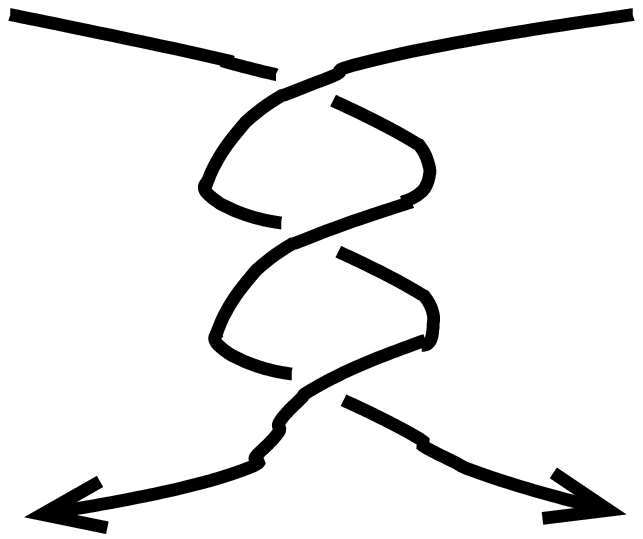}}& \cdots \\ 
 & n=-1 & n=0 & n=1 & n=2 & \\
\end{array}
\end{equation}

In Eisermann \cite{MR1997586}, the idea was generalized to twist lattices $\Phi:\mathbb{Z}^n \to \mathscr{K}$.  It was shown that an invariant is a finite-type invariant of degree $\le n$ if and only if it is a polynomial of degree $\le n$ on every twist lattice. In this paper, twist lattices are generalized, via Gauss diagrams, to long virtual knots. 

Polynomial characterization of the finite-type concept for virtual knots is more complicated. In fact, there are two different flavors of finite-type invariants.  The first is due to Goussarov-Polyak-Viro \cite{MR1763963} and the second is due to Kauffman \cite{virtkauff}.  The GPV notion is purely combinatorial/algebraic and allows for an easily defined universal invariant.  Invariants of Kauffman type include extensions of some classical Vassiliev finite type invariants (like the Jones polynomial). Any GPV finite-type invariant is necessarily a Kauffman finite-type invariant, but it is not yet clear what obstructions there are for a Kauffman finite-type invariant to be GPV.

For each of the two kinds of finite-type, there is a different natural notion of twist lattice. Kauffman finite-type invariants, being closely related to Vassiliev's original definition of finite-type, are associated to twist lattices with only cosmetic alterations from the classical knot case (called regular twist lattices).  The GPV twist lattices are very different from the classical twist lattices (here called fractional twist lattices).  The first result of this paper is to prove the following theorem, which mirrors the theorem of Eisermann:

\begin{theorem} \label{derivthm} An invariant is Kauffman finite-type of degree $\le n$ if and only if it is a polynomial of degree $\le n$ on every regular twist lattice.  An invariant is GPV finite-type of degree $\le n$ if and only if it is a polynomial of degree $\le n$ on every fractional twist lattice.   
\end{theorem}

The second topic of this paper is an investigation of the Jones-Kauffman polynomial, $f_K(A)$ for virtual long knots $K$.  Denote by $v_k:\mathscr{K} \to \mathbb{Q}$ the coefficient of $x^k$ in the power series expansion of $f_K(e^x)$ about $x=0$.  Kauffman has shown that $v_k$ is a Kauffman finite-type invariant of degree $k$.  In this paper, the behavior of these invariants on regular and fractional twist sequences is investigated.

The Jones-Kauffman polynomial has an interesting property in the case of virtual knots.  It is invariant under a ``virtualization move''.  This new move, also discovered by Kauffman, may or may not preserve isotopy class of classical knots.  In this paper, the virtualization move plays a central role.  The proof of the following theorem shows that, to some extent, invariance under the virtualization move is what prohibits the invariants $v_k$ from being of GPV finite-type. In \cite{virtkauff}, it was noted that the following theorem is true in the case $k=2$.

\begin{theorem} \label{jkthm} For all $k \ge 2$, $v_k:\mathscr{K} \to \mathbb{Q}$ is a Kauffman finite type invariant of degree $k$, but is not a GPV finite type invariant of degree $\le n$ for any $n$.  In particular, $v_k$ is not a polynomial on every fractional twist sequence.
\end{theorem}

This paper is organized as follows.  In the remainder of Section 1, the definitions of virtual long knots and the two kinds of finite-type invariants are reviewed.   Also in Section 1, discrete derivatives and discrete power series are developed in a way well suited for application to Gauss diagrams.  In Section 2, the two types of twist lattice are defined and the proof of Theorem \ref{derivthm} is given.  Section 3 contains the proof of Theorem \ref{jkthm} and its supporting lemmas.

\subsection{Long Virtual Knots, Dictionary of Crossing Types}  A long knot diagram, $\tau:\mathbb{R} \to \mathbb{R}^2$, is an immersion of the line which coincides with the usual embedding of $\mathbb{R}\to \mathbb{R}^2$ outside a compact set. Moreover, we require that every point at which $\tau$ is not injective, there is a specification of a crossing over or under in the usual sense.  

To every long knot diagram $\tau:\mathbb{R} \to \mathbb{R}^2$, there is associated a Gauss diagram (see \cite{MR2068425} or \cite{MR1763963}). On $\mathbb{R}$, we mark all points at which $\tau$ is not injective. Suppose $x,y \in \mathbb{R}$ and $x<y$.  If $\tau(x)=\tau(y)$, we connect $x$ and $y$ by an arc. If traveling along the knot from $-\infty \to \infty$ the knot overcrosses at $x$, we embellish the arc between $x$ and $y$ with an arrowhead pointing right.  If the the overcrossing occurs at $y$, we draw an arrowhead pointing left. 

Recall that a crossing is said to be right-handed if it locally resembles the diagram in Equation 1, $n=1$ and left-handed if it locally resembles the diagram in Equation 1, $n=0$. Each arc on the Gauss diagram of a long knot is given a sign of $+$ if the corresponding crossing is right-handed or a $-$ if the corresponding crossing is left-handed.  As an example, the diagram in Figure \ref{testseq} for $n=3$ is a Gauss diagram of the long right-handed trefoil knot. Gauss diagrams are considered equivalent up to orientation preserving homeomorphism of $\mathbb{R}$ that maps arcs to arcs while preserving both direction and sign.
\begin{figure}[b]
\[
\begin{array}{ccc} \xymatrix{ \scalebox{.15}{\psfig{figure=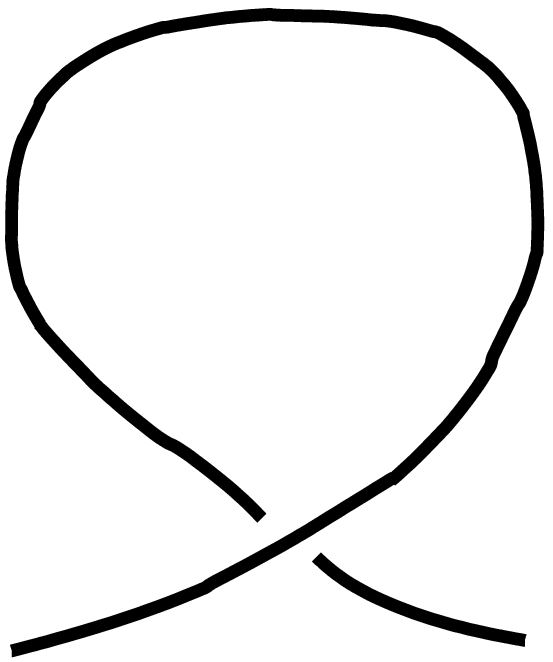}} \ar[r]^{\text{RI}} & \ar[l] \scalebox{.15}{\psfig{figure=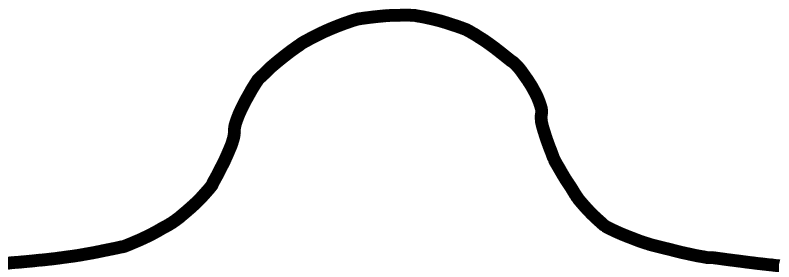}}}& \xymatrix{ \scalebox{.15}{\psfig{figure=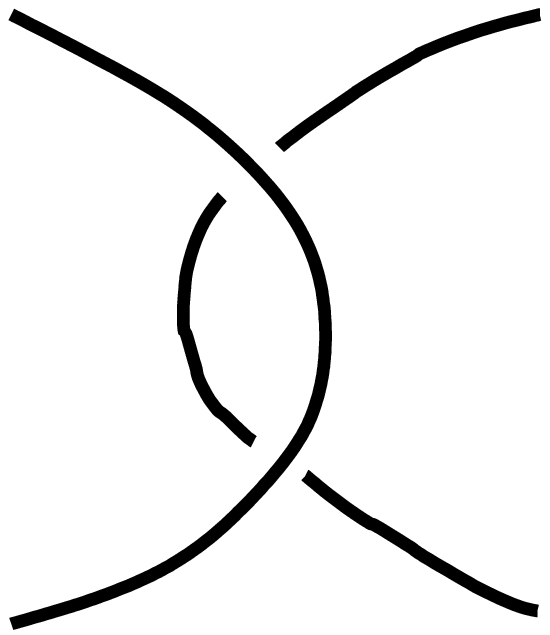}} \ar[r]^{\text{RII}} & \ar[l] \scalebox{.15}{\psfig{figure=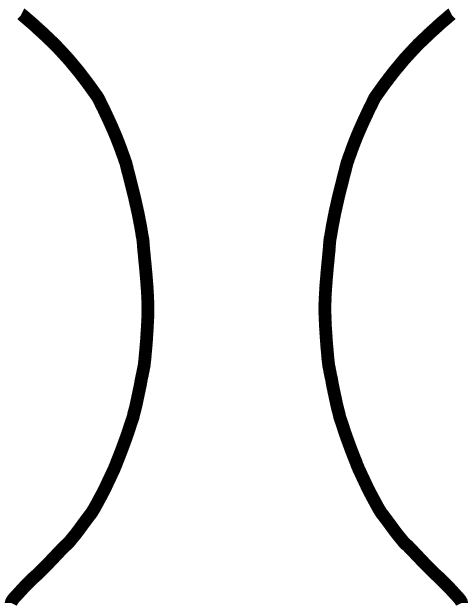}}} & \xymatrix{ \scalebox{.15}{\psfig{figure=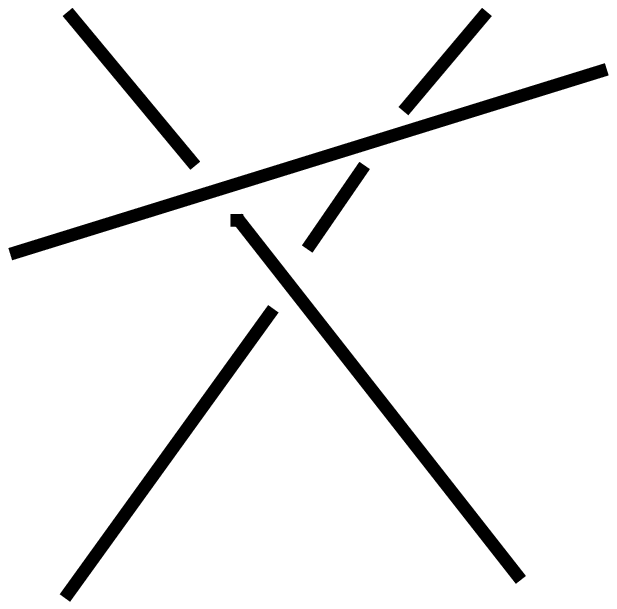}} \ar[r]^{\text{RIII}} & \ar[l] \scalebox{.15}{\psfig{figure=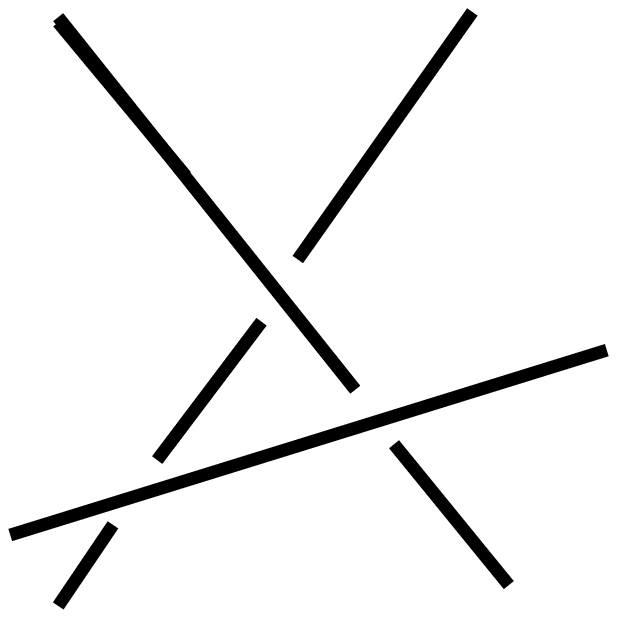}}} \end{array}
\]
\[
\begin{array}{ccc} \xymatrix{ \scalebox{.15}{\psfig{figure=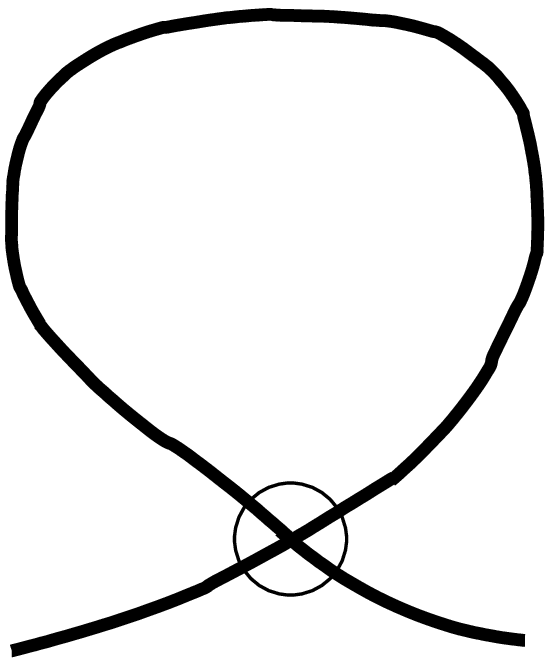}} \ar[r]^{\text{VrI}} & \ar[l] \scalebox{.15}{\psfig{figure=riiright.eps}}}& \xymatrix{ \scalebox{.15}{\psfig{figure=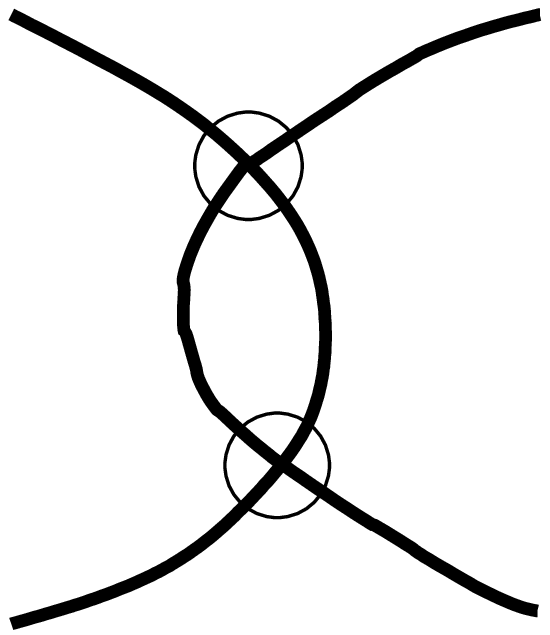}} \ar[r]^{\text{VrII}} & \ar[l] \scalebox{.15}{\psfig{figure=r2right.eps}}} & \xymatrix{ \scalebox{.15}{\psfig{figure=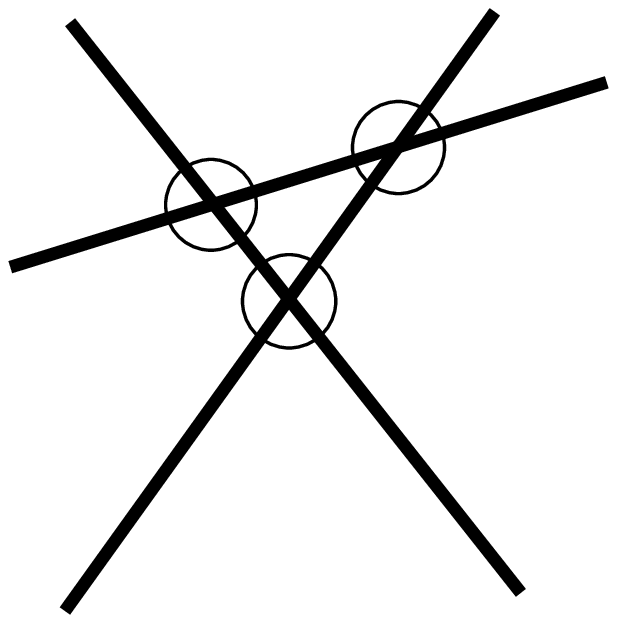}} \ar[r]^{\text{VrIII}} & \ar[l] \scalebox{.15}{\psfig{figure=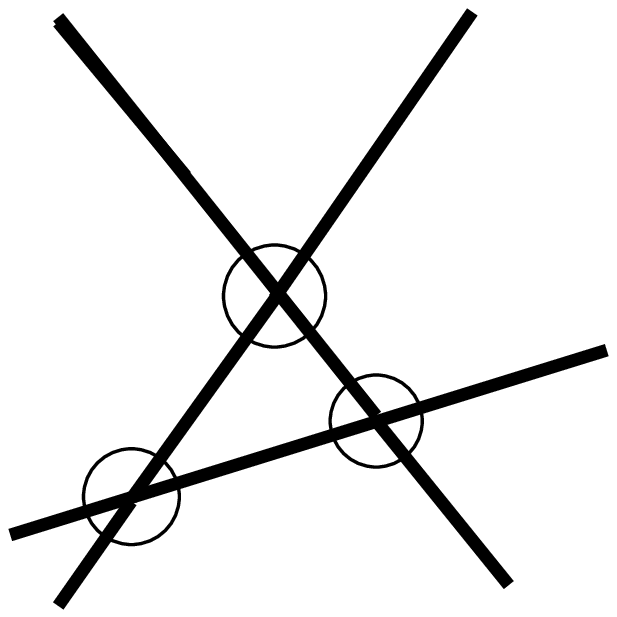}}} \end{array}
\]
\[
\xymatrix{ \scalebox{.15}{\psfig{figure=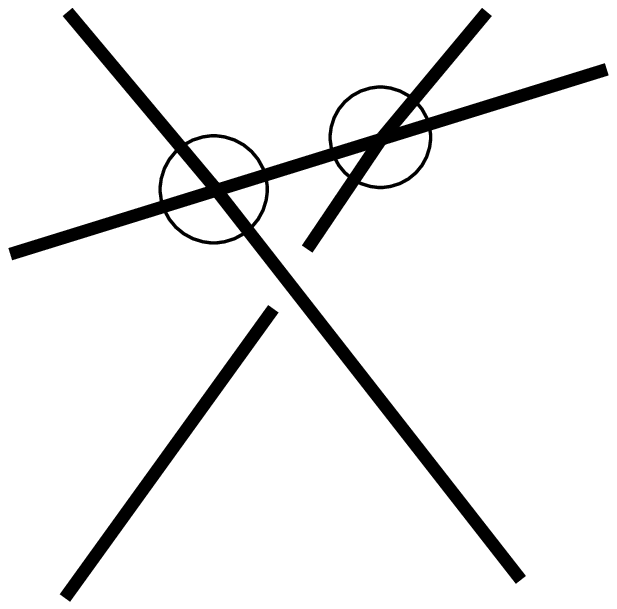}} \ar[r]^{\text{VrIV}} & \ar[l] \scalebox{.15}{\psfig{figure=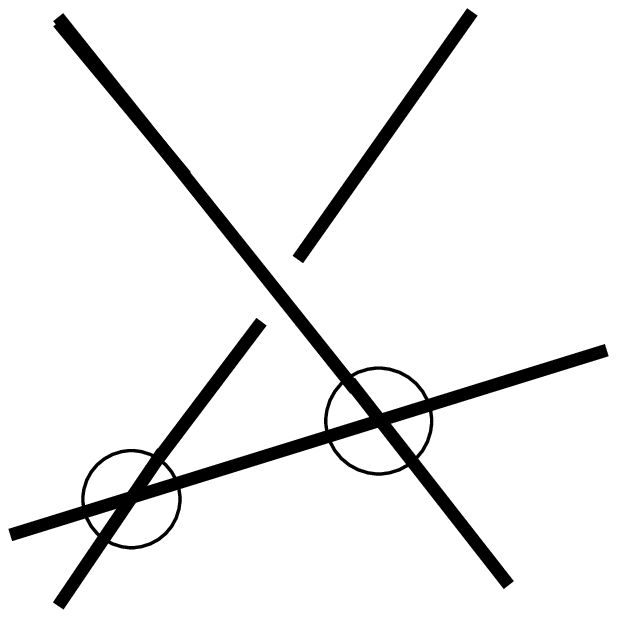}}}
\]
\caption{The Reidemiester and Virtual Moves} \label{virtmoves}
\end{figure}

A virtual long knot diagram is defined similarly to a long knot diagram.  The difference is that at each point at which the immersion is not injective there are three kinds of embellishments allowed: overcrossing, under crossing, or a four valent graphical vertex.  The graphical vertex is denoted with a small circle surrounding the transversal intersection of two arcs.  Long virtual knots are equivalent up to the three Reidemeister moves and four additional moves.  They are shown in Figure \ref{virtmoves}.

The Gauss diagram of a long virtual knot is obtained by including all the classical crossings in the usual way and then ignoring the virtual crossings.  It is known that every Gauss diagram determines the unique equivalence class of a virtual long knot. The set of equivalence classes of long virtual knots will be denoted $\mathscr{K}$ and the set of equivalence classes of Gauss diagrams will be denoted $\mathscr{D}$.

In addition to the classical and virtual crossings, there are also singular crossings and semi-virtual crossings. Each is associated to a specific filtration of $\mathbb{Z}[\mathscr{K}]$.  The virtual knot diagram notation, Gauss diagram notation, and defining relations are summarized in Table \ref{crossdict}. Note that the orientations on $\mathbb{R}$ and $\mathbb{R}^2$ allow for a crossing of any type to have a local orientation.  These are denoted in the Gauss diagram by a $+$ or $-$ sign. They are typically placed on $\mathbb{R}$ at an arrow endpoint or somewhere just above the arrow itself. In equations, a variable crossing sign is typically denoted with a $\varepsilon$.
\begin{figure}[t] 
\[
\begin{tabular}{|c|c|c|} \hline
& Virtual Knot Diagram & Gauss Diagram \\ \hline
Semi-virtual & $ \begin{array}{c} \scalebox{.12}{\psfig{figure=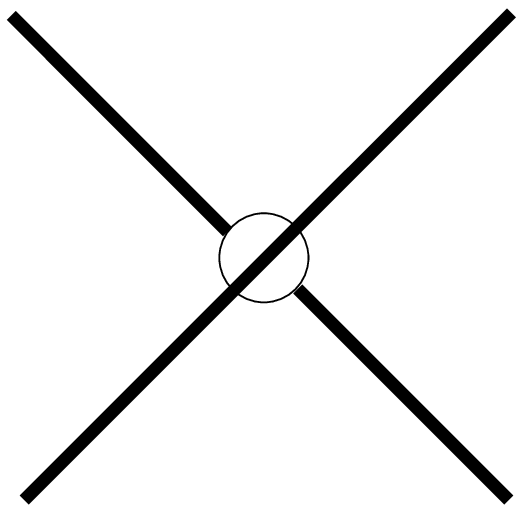}} \\ \end{array} = \begin{array}{c} \scalebox{.12}{\psfig{figure=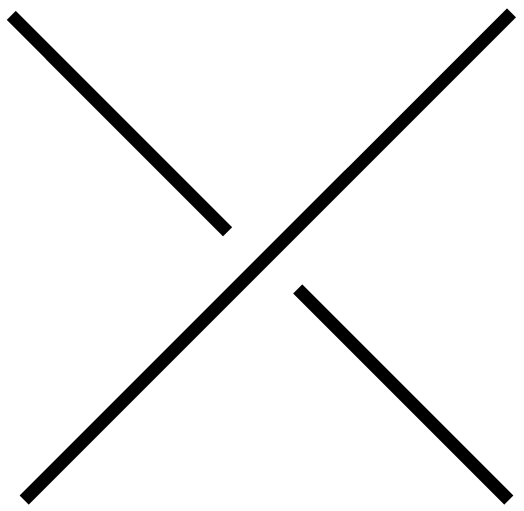}} \\ \end{array}-\begin{array}{c} \scalebox{.12}{\psfig{figure=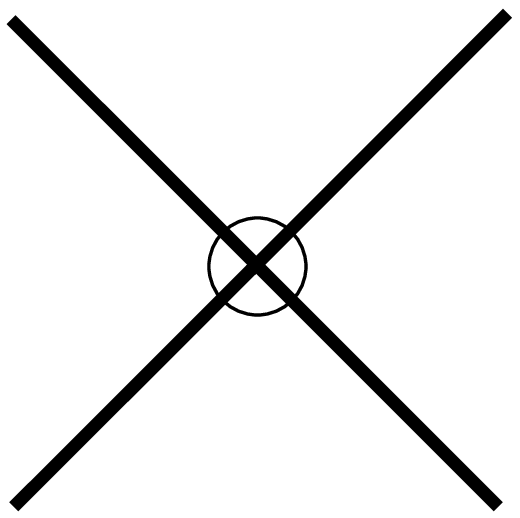}} \\\end{array}
$ & $
\begin{array}{c} \scalebox{.12}{\psfig{figure=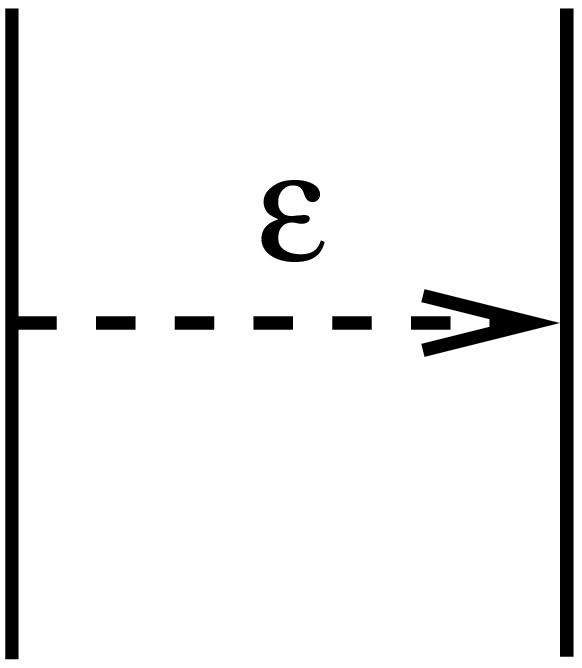}} \\ \end{array} =  \begin{array}{c} \scalebox{.12}{\psfig{figure=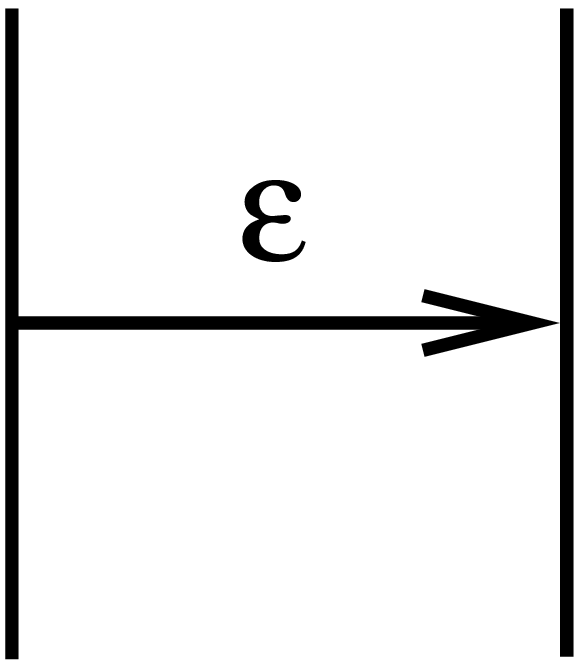}} \\ \end{array}-\begin{array}{c} \scalebox{.12}{\psfig{figure=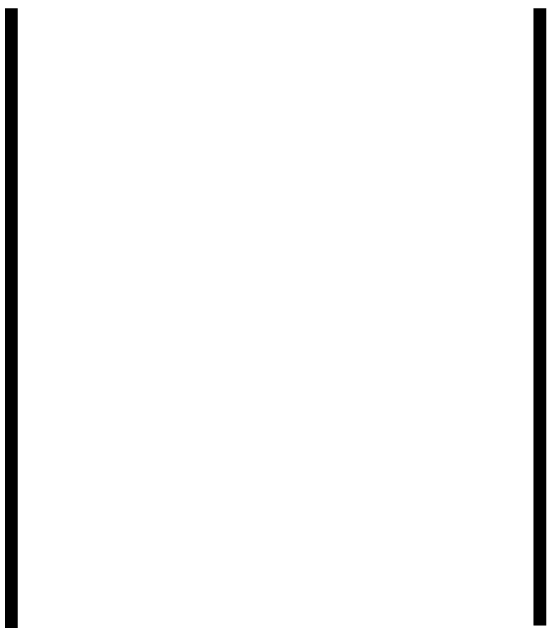}} \\\end{array}
$ \\ \hline
Singular & $\begin{array}{c} \scalebox{.12}{\psfig{figure=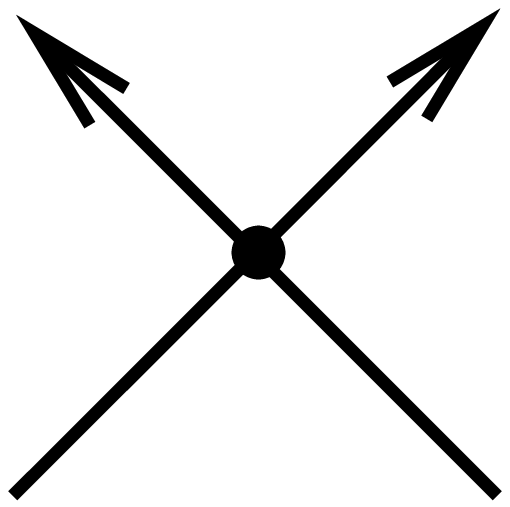}} \\ \end{array} = \begin{array}{c} \scalebox{.12}{\psfig{figure=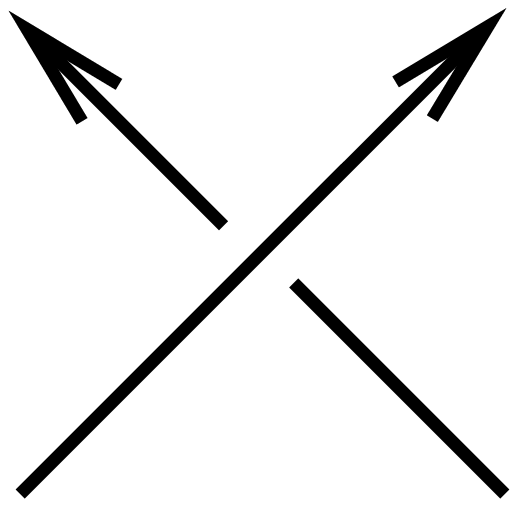}} \\ \end{array}-\begin{array}{c} \scalebox{.12}{\psfig{figure=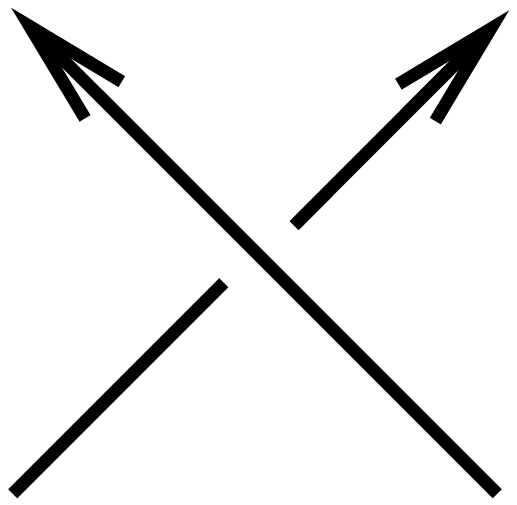}} \\\end{array} $ & $ \begin{array}{c} \scalebox{.12}{\psfig{figure=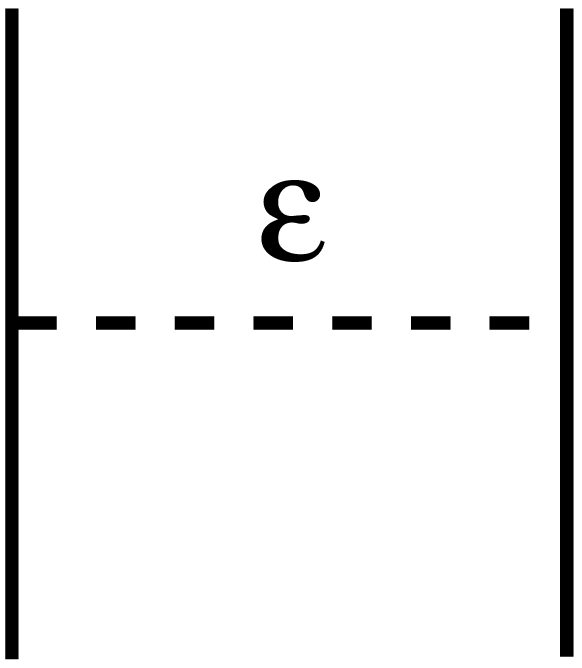}} \\ \end{array} = \varepsilon \cdot \begin{array}{c} \scalebox{.12}{\psfig{figure=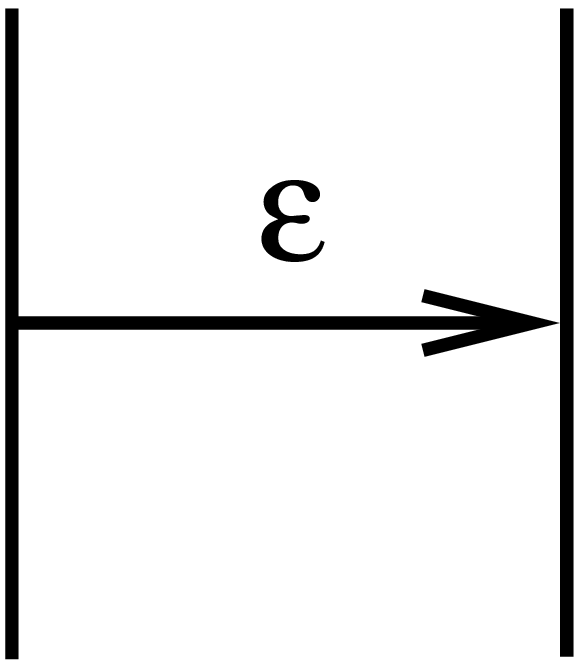}} \\ \end{array}-\varepsilon \cdot \begin{array}{c} \scalebox{.12}{\psfig{figure=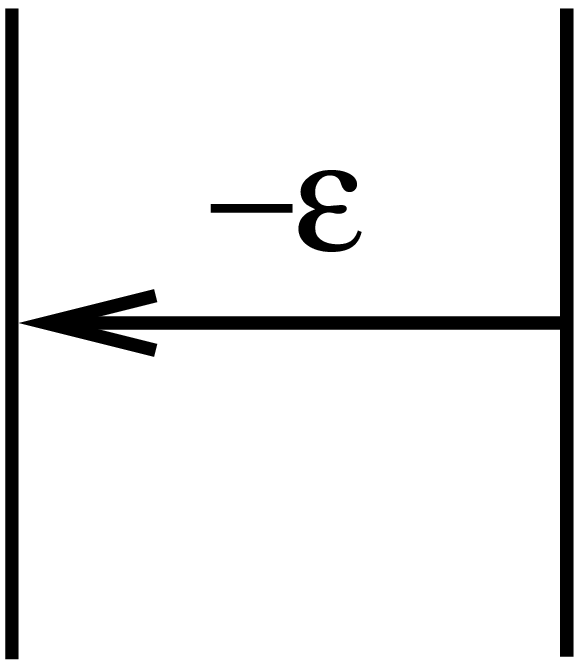}} \\\end{array}
$ \\ \hline
\end{tabular}
\]
\caption{Dictionary of Crossing Types} \label{crossdict}
\end{figure}
The two types of crossings are related by the following local equation:
\begin{equation}\label{virtandsing}
\begin{array}{c} \scalebox{.15}{\psfig{figure=doublepoint.eps}} \\ \end{array} = \begin{array}{c} \scalebox{.15}{\psfig{figure=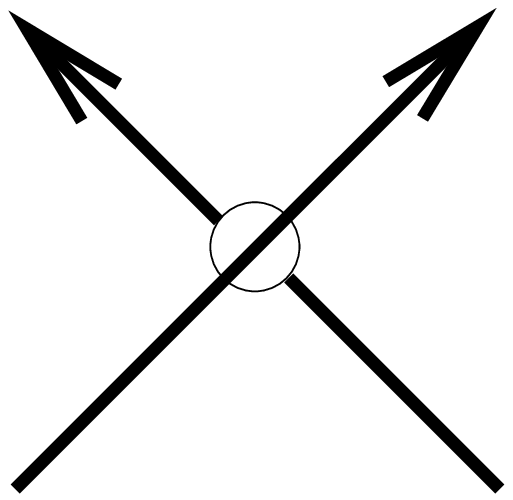}} \\ \end{array}-\begin{array}{c} \scalebox{.15}{\psfig{figure=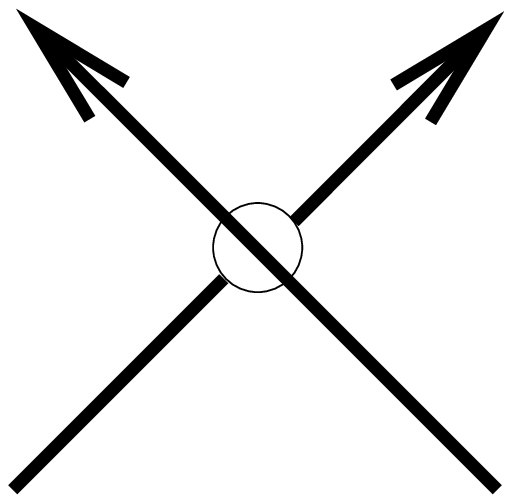}} \\\end{array}
\end{equation}
\subsection{Two Flavors of Finite Type} Let $v:\mathscr{K} \to \mathbb{Q}$ be a virtual knot invariant.  It can be extended to $\mathbb{Z}[\mathscr{K}]$ by linearity. Moreover, the relations in Table \ref{crossdict} allow us to extend any such invariant uniquely to diagrams with semi-virtual and singular crossings (see \cite{MR1763963}).  For classical knots, a Vassiliev (or finite-type) invariant of degree $\le n$ is an invariant  which vanishes after $n+1$ or more extensions to diagrams with singular crossings.  

For virtual knots, there are two notions of finite-type invariant.  A Goussarov-Polyak-Viro (GPV) finite-type invariant of degree $\le n$ is one which vanishes on any element of $\mathbb{Z}[\mathscr{K}]$ represented with $n+1$ or more semi-virtual crossings (see \cite{MR1763963}).  A Kauffman finite-type invariant of degree $\le n$ is an invariant preserving rigid vertex isotopy that vanishes on any diagram with $n+1$ or more singular crossings (see \cite{virtkauff}).  Equation \ref{virtandsing} shows that any GPV finite-type invariant of degree $\le n$ is a Kauffman finite-type invariant of degree $ \le n$.

GPV finite-type invariants admit a purely algebraic characterization. Let $\mathscr{A}$ denote the free abelian group generated by Gauss diagrams containing only dashed arrows.  The Polyak algebra, denoted $\mathscr{P}$, is the quotient of $\mathscr{A}$ by the submodule $\mathscr{R}$ generated by the three classes of localized relations shown in Figure \ref{polyakrels}.
\begin{figure}  
\[
\begin{array}{c}\scalebox{.15}{\psfig{figure=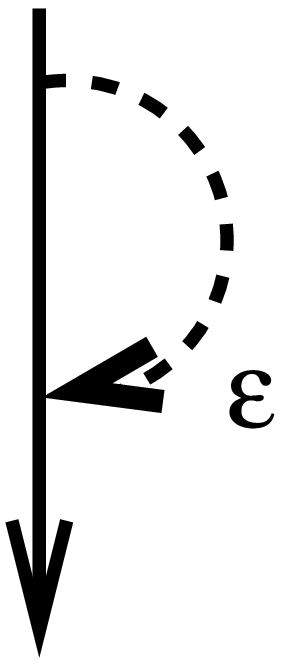}} \end{array} =0, \begin{array}{c}\scalebox{.15}{\psfig{figure=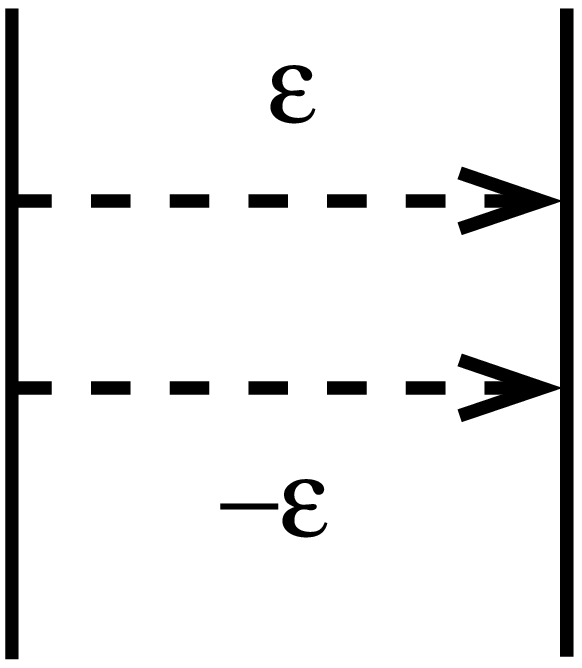}} \end{array}+\begin{array}{c}\scalebox{.15}{\psfig{figure=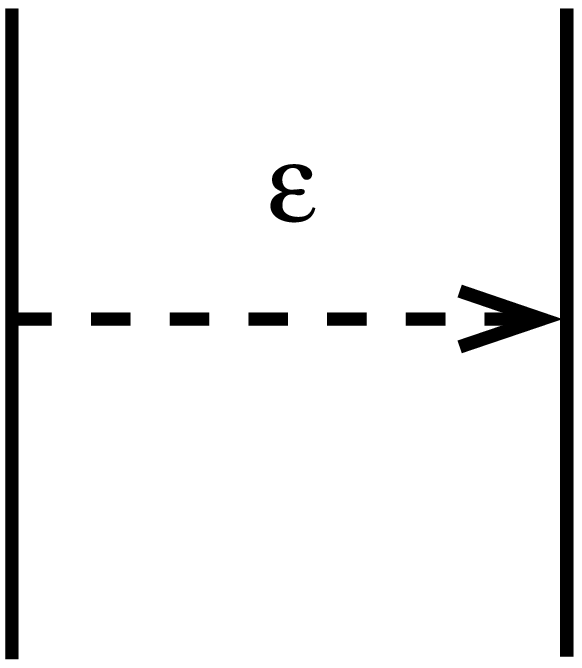}} \end{array}+\begin{array}{c}\scalebox{.15}{\psfig{figure=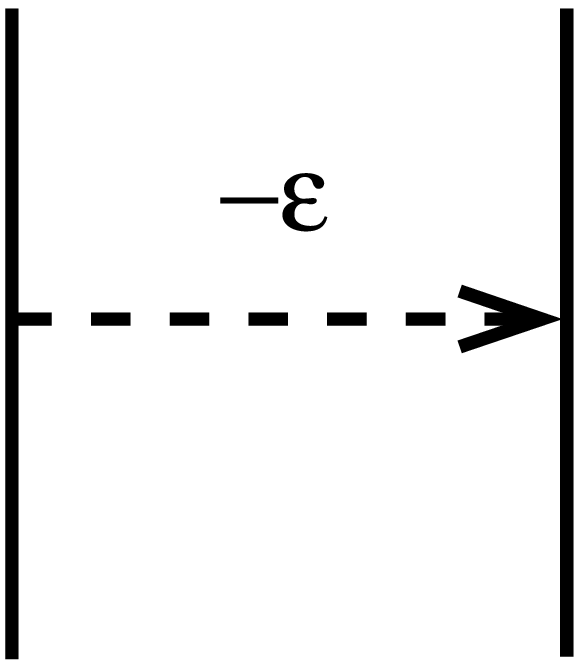}} \end{array}=0,
\]
\begin{eqnarray*}
\begin{array}{c}\scalebox{.15}{\psfig{figure=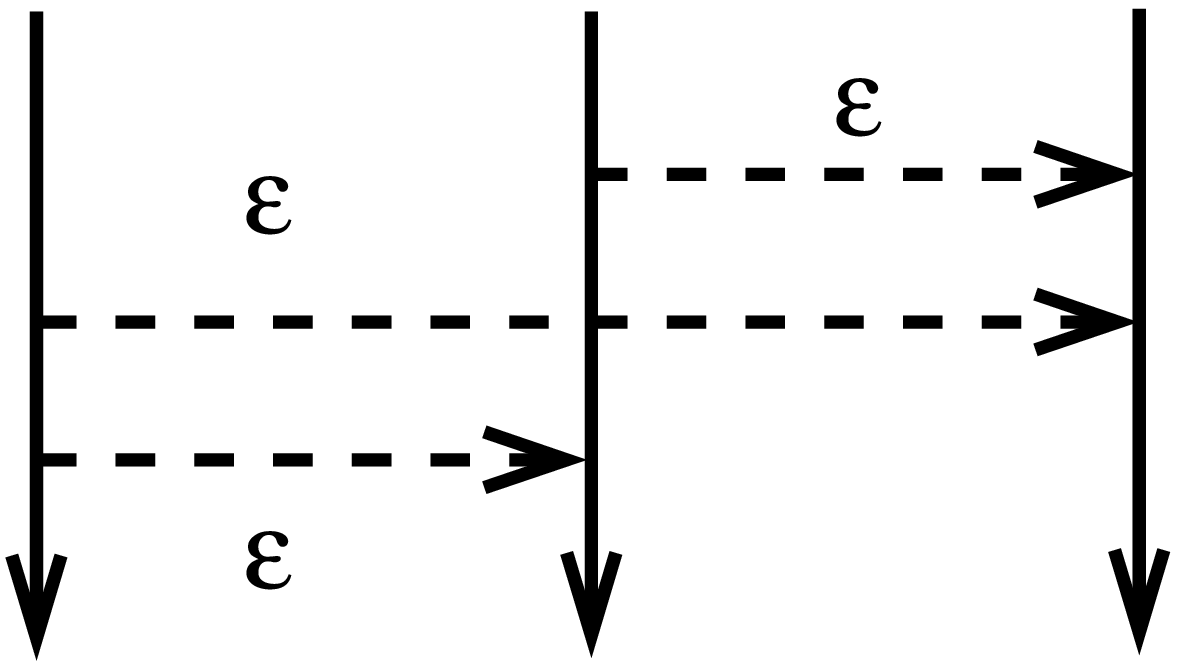}} \end{array}+\begin{array}{c}\scalebox{.15}{\psfig{figure=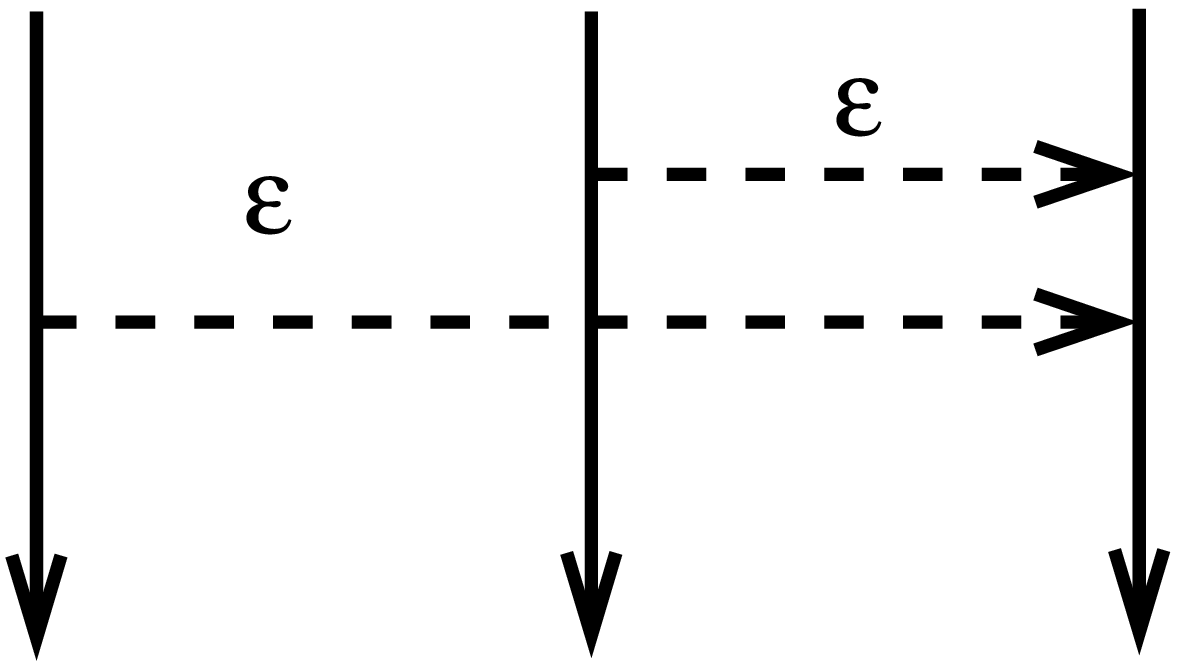}} \end{array}+\begin{array}{c}\scalebox{.15}{\psfig{figure=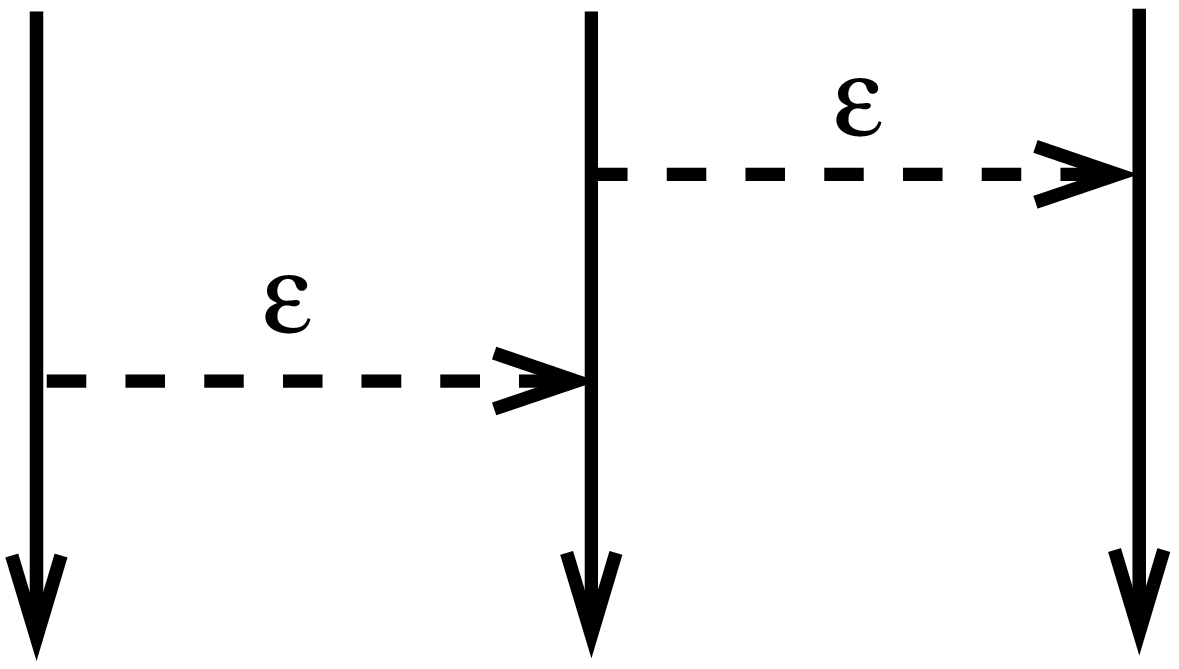}} \end{array}+\begin{array}{c}\scalebox{.15}{\psfig{figure=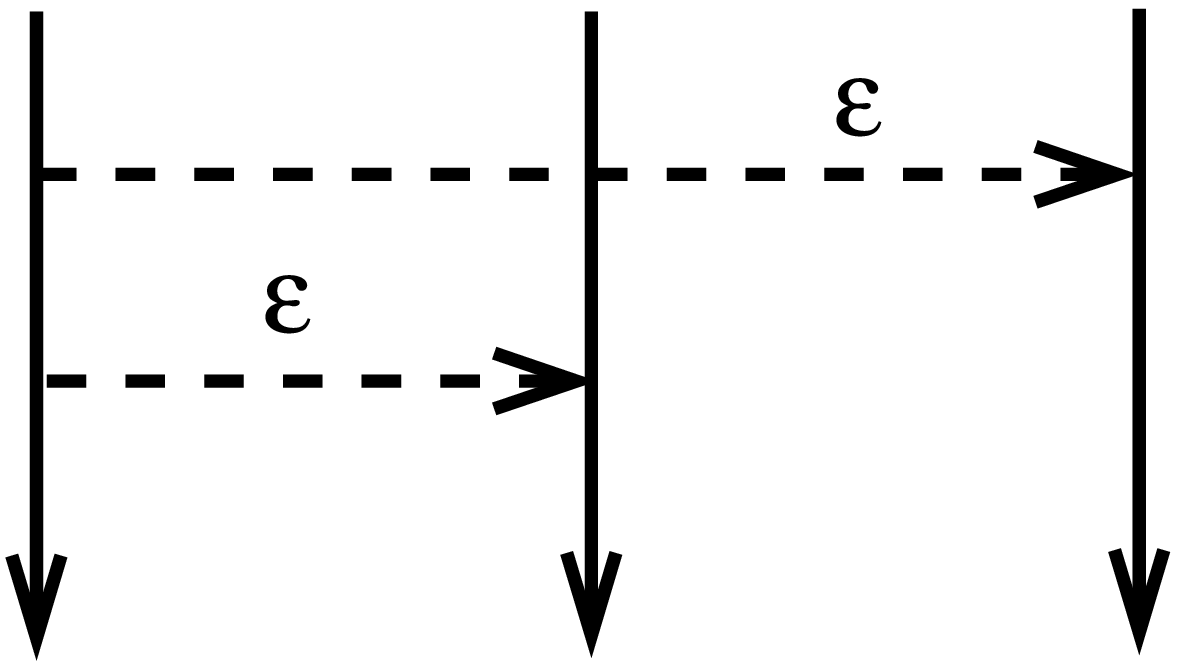}} \end{array} &=& \\ \begin{array}{c}\scalebox{.15}{\psfig{figure=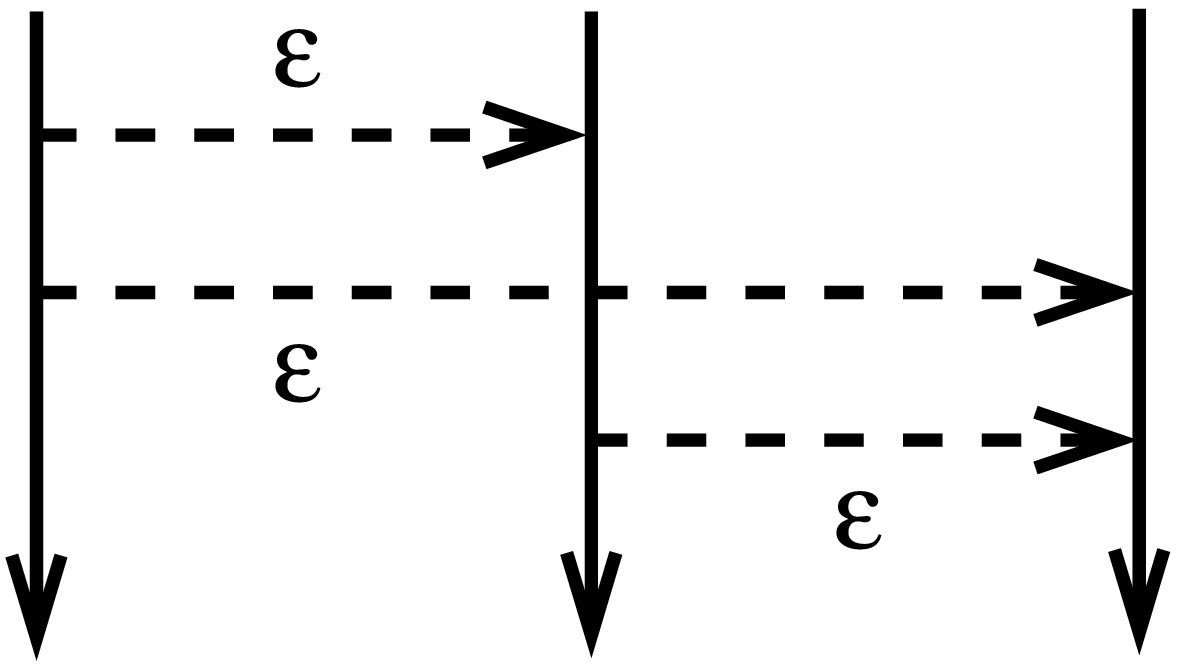}} \end{array}+\begin{array}{c}\scalebox{.15}{\psfig{figure=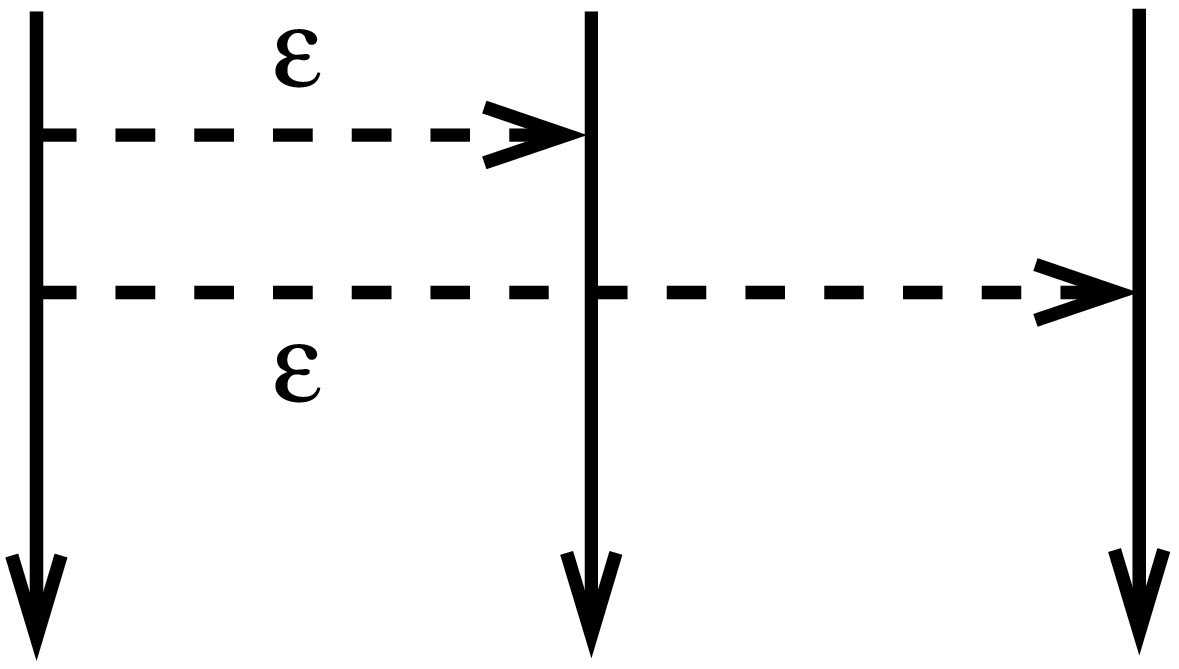}} \end{array}+\begin{array}{c}\scalebox{.15}{\psfig{figure=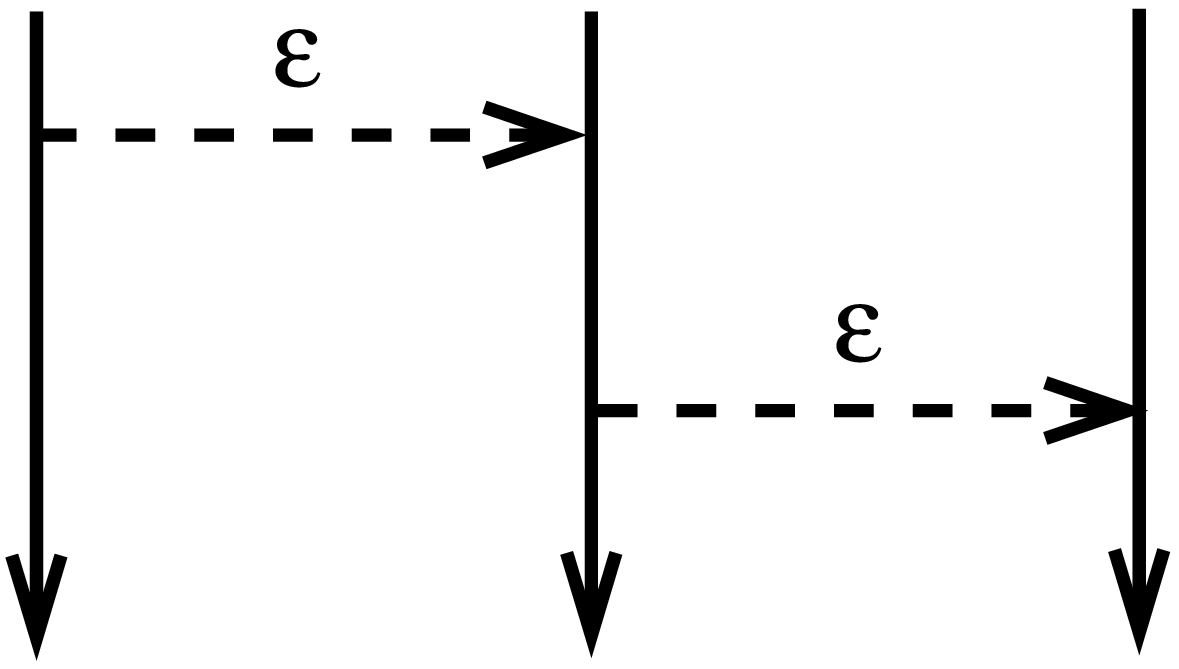}} \end{array}+\begin{array}{c}\scalebox{.15}{\psfig{figure=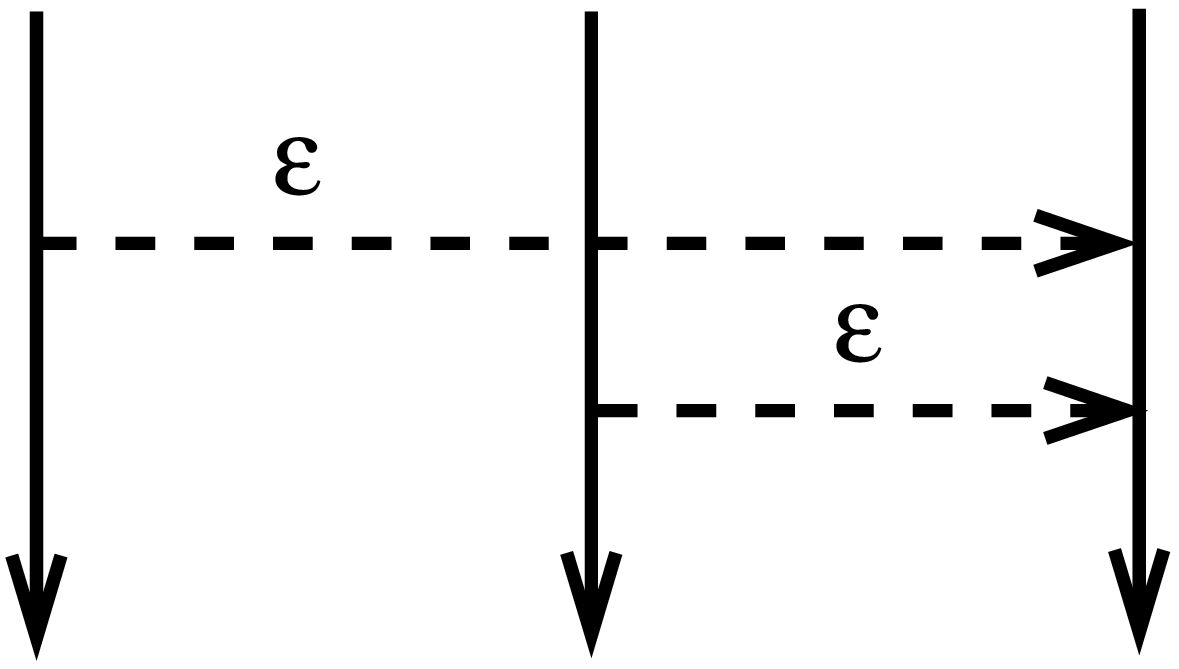}} \end{array} & & \\
\end{eqnarray*}
\caption{Polyak Relations} \label{polyakrels}
\end{figure}

The universal knot invariant for GPV finite-type invariants is surprisingly pleasant to define. It plays a crucial role in the polynomial characterization of these invariants. Let $i:\mathscr{D} \to \mathscr{A}$ to be the map which makes dashed every arrow of a Gauss diagram.  Define $I:\mathbb{Z}[D]\to \mathscr{A}$ by the equation:
\[
I(D)=\sum_{D' \subset D} i(D')
\]
The sum is taken over all Gauss diagrams $D'$ obtained from $D$ by deleting a subset of the arrows of $D$.  
\begin{theorem}[Goussarov, Polyak, Viro, \cite{MR1763963}] \label{iisom} The map $I:\mathbb{Z}[\mathscr{D}] \to \mathscr{A}$ is an isomorphism.  The inverse can be defined explicitly:
\[
I^{-1}(A)=\sum_{A'\subset A} (-1)^{|A-A'|} i^{-1}(A)
\]
Here, $|A- A'|$ means the number of arrows in $A$ that are not in $A'$.  Furthermore, if $D \in \mathbb{Z}[\mathscr{D}]$ has dashed arrows, then every element in the sum defining $I(D)$ also has every dashed arrow of $D$.  Finally, the map extends to an isomorphism of the quotient algebras $I:\mathbb{Z}[\mathscr{K}] \to \mathscr{P}$.
\end{theorem}
Denote by $\mathscr{A}_n$ the submodule of all diagrams having more than $n$ arrows.  Let $\mathscr{P}_n=\mathscr{A}/(\mathscr{A}_n+\mathscr{R})$. The map $I_n:\mathbb{Z}[\mathscr{K}] \to \mathscr{P} \to \mathscr{P}_n$ is the above isomorphism composed with the projection.
\begin{theorem} [Goussarov, Polyak, Viro, \cite{MR1763963}] \label{GPVuni} The map $I_n:\mathbb{Z}[\mathscr{K}] \to \mathscr{P} \to \mathscr{P}_n$ is universal in the sense that if $G$ is any abelian group, and $v$ is a GPV finite-type invariant of type $\le n$, then there is a map $\bar{v}:\mathscr{P}_n \to G$ such that the following diagram commutes:
\[
\xymatrix{\mathbb{Z}[\mathscr{K}] \ar[r]^v \ar[d]_I & G \\
\mathscr{P} \ar[r]_{\varphi_n} \ar[ur]_{v I^{-1}} & \mathscr{P}_n \ar@{-->}[u]_{\bar{v}} \\
}
\]
In particular, the vector space of rational valued invariants of type $\le n$ is finite dimensional and can be identified with $\text{Hom}_{\mathbb{Z}}(\mathscr{P}_n,\mathbb{Q})$.
\end{theorem}

\subsection{Discrete Derivatives, Notation, Polynomials} The material in this section is essentially that of the classic umbral calculus. Different forms of it appear in many papers. Some additional care is taken in their definition here because our twist sequences must also account for the signs of the crossings.

Suppose that $A$ is an abelian group and $\Phi:\mathbb{Z}\to A$ a function.  We define two discrete derivatives of $\Phi$, denoted $\partial^+\Phi:\mathbb{Z} \to A$ and $\partial^-\Phi:\mathbb{Z} \to A$, as follows: 
\[
(\partial^+\Phi)(z)=\Phi(z+1)-\Phi(z)
\]
\[
(\partial^-\Phi)(z)=\Phi(z)-\Phi(z-1)
\]
Now, let $n \in \mathbb{N}$ be fixed and $\Phi:\mathbb{Z}^n \to A$.  For each $i$, $1 \le i \le n$, we define partial discrete derivatives with respect to each coordinate, denoted $\partial_i^{\nu}:\mathbb{Z}^n \to A$, where $\nu=\pm$:
\[
(\partial_i^+\Phi)(z_1,\ldots,z_i,\ldots,z_n)=\Phi(z_1,\ldots,z_i+1,\ldots,z_n)-\Phi(z_1,\ldots,\; z_i,\ldots,z_n)
\]
\[
(\partial_i^-\Phi)(z_1,\ldots,z_i,\ldots,z_n)=\Phi(z_1,\ldots,z_i,\ldots,z_n)-\Phi(z_1,\ldots,\; z_i-1,\ldots,z_n)
\]
Discrete derivatives can be iterated in the obvious way.  An immediate consequence is that $\partial_i^{\mu} \partial_j^{\nu}=\partial_{j}^{\nu} \partial_i^{\mu}$. Define $\partial_i^0\Phi(z)=\Phi(z)$. If $a \in \mathbb{N}$ and $z \in \mathbb{Z}^n$, we define higher derivatives inductively by $\partial^a_i\Phi(z)=\partial^+_i(\partial^{a-1}_i\Phi)(z)$ and $\partial^{-a}_i\Phi(z)=\partial^-_i(\partial_i^{-a+1}\Phi)(z)$.  Now for $\alpha=(\alpha_1,\ldots,\alpha_n) \in \mathbb{Z}^n$, we define:
\[
(\partial^{\alpha}\Phi)(z)=(\partial^{\alpha_1}_1 \partial^{\alpha_2}_2 \cdots \partial^{\alpha_n}_n \Phi)(z)
\]
\begin{lemma} \label{discderiv} Let $\alpha \in \mathbb{N} \cup\{0\}$ and $\Phi:\mathbb{Z} \to A$. Then we have the following formulas:
\[
\partial^{\alpha} \Phi(z)=\sum_{k=0}^{\alpha}(-1)^{\alpha+k} {\alpha \choose k} \Phi(z+k)
\]
\[
\partial^{-\alpha} \Phi(z)=\sum_{k=0}^{\alpha}(-1)^{k} {\alpha \choose k} \Phi(z-k)
\] 
\end{lemma}
\begin{proof} The result follows easily from the binomial theorem and Pascal's triangle.
\end{proof}
From discrete derivatives, one may obtain discrete power series.  Let $n \in \mathbb{N}$.  Define $\mathbb{Z}_2=\{+,-\}$ to be the group on two elements.  We interpret $\nu =(\nu_1,\nu_2,\ldots,\nu_n)\in \mathbb{Z}_2^n$ to be a choice of a positive or negative derivative in each coordinate. Define: 
\[
\mathbb{Z}^+=\{z \in \mathbb{Z}:z\ge 0\}, \; \mathbb{Z}^-=\{z \in \mathbb{Z}:z \le 0\}
\]
If $\nu \in \mathbb{Z}_2^n$, define $\mathbb{Z}^{\nu}=\mathbb{Z}^{\nu_1}\times \ldots \times \mathbb{Z}^{\nu_n}$.

Let $\nu\in\mathbb{Z}_2$ be given and $z \in \mathbb{Z}^{\nu}$. Let $\alpha\in\mathbb{Z}^+$. The discrete power series differs from a normal power series in that powers of the variable $z$ are replaced with $z^{\nu \alpha}$:
\[
z^{\nu\alpha}=\left\{ \begin{array}{cl} z(z-1)(z-2)\cdots(z-\alpha+1) &:\; \nu=+ \\ 
                                    z(z+1)(z+2)\cdots(z+\alpha-1) &: \;\nu=- 
\end{array}\right.
\]
If $\alpha=0$, we set $z^{\nu\alpha}=1$.  This defines a function from $\mathbb{Z}^{\nu} \to \mathbb{Z}$. Now let $n \in \mathbb{N}$ and $\nu=(\nu_1,\ldots,\nu_n)\in \mathbb{Z}_2^n$.  For $\alpha=(\alpha_1,\ldots,\alpha_n)\in (\mathbb{Z}^{+})^n$, we abuse notation and write $\nu \alpha=(\nu_1\alpha_1,\ldots,\nu_n\alpha_n)$.  The notion of powers for lattices of dimension $n>1$ is given by a function $z^{\nu\alpha}:\mathbb{Z}^{\nu}\to\mathbb{Z}$. The function is defined for each $z=(z_1,\ldots,z_n) \in \mathbb{Z}^\nu$ via the equation:
\[
z^{\nu \alpha}=z_1^{\nu_1\alpha_1}z_2^{\nu_2 \alpha_2}\cdots z_n^{\nu_n \alpha_n}
\]

For $\alpha=(\alpha_1,\ldots,\alpha_n)\in (\mathbb{Z}^+)^n$, define $|\alpha|=\alpha_1+\ldots+\alpha_n$. This norm should be interpreted as the degree of $z^{\nu \alpha}$. Also, it is convenient to define $\alpha!=\alpha_1!\alpha_2!\cdots\alpha_n!$

Finally, enough notation is in place to define the discrete power series itself.  Let $n \in \mathbb{N}$.  Let $\Phi:\mathbb{Z}^n \to \mathbb{Q}$ be a function. Let $\nu\in\mathbb{Z}_2^n$ be given and $z\in\mathbb{Z}^{\nu}$.  The $\nu$-power series of $\Phi$, denoted $\Sigma_{\Phi}^{\nu}:\mathbb{Z}^{\nu}\to \mathbb{Q}$, is defined to be:
\[
\Sigma_{\Phi}^{\nu}(z)=\lim_{n\to\infty} \sum_{0 \le |\alpha|\le n} \frac{(\partial^{\nu \alpha}\Phi)(\vec0)}{\alpha!} z^{\nu \alpha}
\]
For example, if $n=2$, we have a different series defined in each quadrant of $\mathbb{Z} \times \mathbb{Z}$.  For $n=3$, there is a different definition in each octant.

Having defined a power series for every element of $\mathbb{Z}_2^n$, it is a simple matter to define the power series $\Sigma_{\Phi}:\mathbb{Z}^n\to A$.  For $z \in \mathbb{Z}^{\nu}$, define $\Sigma_{\Phi}(z)=\Sigma_{\Phi}^{\nu}(z)$. Since $(\partial_i^0\Phi)(0)=\Phi(0)$ and $z^{\pm 0}=1$, $\Sigma_{\Phi}^{\nu}$ is well-defined whenever the above limit is defined.
\begin{lemma}Let $\Phi:\mathbb{Z} \to \mathbb{Q}$ be a function.  Let $\nu \in \mathbb{Z}_2$. The expression $\Sigma_{\Phi}^{\nu}(z)$ is defined for all $z \in \mathbb{Z}^{\nu}$.  Moreover, $\Sigma_{\Phi}^{\nu}(z)=\Phi(z)$ for all $z \in \mathbb{Z}^{\nu}$. 
\end{lemma}
\begin{proof} This is a consequence of the binomial theorem and the properties of the binomial coefficients
\end{proof}

\begin{theorem} \label{powserthm} Let $n \in \mathbb{N}$ and $\Phi:\mathbb{Z}^n \to \mathbb{Q}$ be a function.  Let $\nu \in \mathbb{Z}_2^n$. The expression $\Sigma_{\Phi}^{\nu}(z)$ is defined for all $z \in \mathbb{Z}^{\nu}$.  Moreover, $\Sigma_{\Phi}^{\nu}(z)=\Phi(z)$ for all $z \in \mathbb{Z}^{\nu}$. Thus, $\Sigma_{\Phi}(z)=\Phi(z)$ for all $z \in \mathbb{Z}^n$.
\end{theorem}
\begin{proof} This is a consequence of the previous lemma and properties of the binomial coefficients.
\end{proof}
In the discrete case, a function $f:\mathbb{Z}^n \to A$ is said to be a polynomial of degree $\le m$ if all its discrete derivatives of total degree $>m$ vanish. The above theorem establishes a connection between discrete polynomials and regular polynomials $\mathbb{R}^n \to \mathbb{R}$. 
%In fact, it is easy to see that for $\alpha \in \mathbb{Z}^+$, the polynomials $z^{\pm \alpha}, z^{\pm (\alpha-1)},\ldots, 1$ form a linearly independent set over $\mathbb{R}$.  Theorem \ref{powserthm} shows that they span the vector space of polynomials of degree $\le \alpha$. Writing out any polynomial in the basis $z^{\pm \alpha}, z^{\pm (\alpha-1)},\ldots, 1$ gives the values of the discrete derivatives.  All but finitely many of the discrete derivatives are zero, hence:
\begin{observation} \label{obs1} If $f:\mathbb{R}^n \to \mathbb{R}$ is a polynomial (in the usual sense) of degree $\le m$, then for every $\alpha \in (\mathbb{Z}^+)^n$, $\nu \in \mathbb{Z}_2^n$, if $|\alpha|>m$, then $\partial^{\nu\alpha}f(\vec{0})=0$. 
\end{observation}
\begin{proof}  Consider the $n=1$ case.  The set $z^{\pm \alpha}, z^{\pm (\alpha-1)},\ldots, 1$ of $\alpha$ polynomials is a linearly independent spanning set for the vector space of one variable polynomials of degree $\le \alpha$. Hence all discrete derivatives of order $> \alpha$ vanish. A similar argument may be used to establish the $n>1$ case. 
\end{proof}
\section{Twist Lattices}
\subsection{Regular and Fractional Twist Lattices} For knot invariants, local coordinates are obtained through twist sequences and twist lattices (see \cite{MR1997586}).  Here, they are defined via their Gauss diagrams.  

A proper pair of a Gauss diagram is a pair $(A,A')$, $A <A'$, of disjoint open intervals in $\mathbb{R}$ such that any chord or arrow with an endpoint in $A \cup A'$ has both endpoints in $A \cup A'$ and both endpoints are not in the same interval. 

A function $\Phi: \mathbb{Z} \to \mathscr{K}$ is called a regular twist lattice of dimension one (or twist sequence) if every long virtual knot in the image of $\Phi$ has an identical Gauss diagram presentation except in a proper pair $(A,A')$.  Inside the intervals $A$ and $A'$, $\Phi(k)$ is given by one of the rows in Figure \ref{regtwist} labeled $XYZ$, where $X=O$ (odd number of arrows) or $X=E$ (even number of arrows). In each column, the left vertical arrow corresponds to the interval $A$ in $\mathbb{R}$ and the right vertical arrow corresponds to the interval $A'$ in $\mathbb{R}$.  The right vertical arrow points upwards or downwards for all diagrams in the row $XYZ$.  Consequently, the ordering of the arrow endpoints in $A'$ is either the same ($S$) or the reverse ($B$) of the arrow endpoints in $A$.  In the table, the variable $Y$ may be either $S$ or $B$. The variable $Z$ corresponds to whether $\Phi(1)$ points left($L$) or right($R$). Thus, there are eight possible kinds of regular twist sequences.  

In the literature, the case that $Y=S$ is often called a vertical twist sequence while the case $Y=B$ is often called a horizontal twist sequence (for example, see \cite{MR1657727}).

A fractional twist lattice of dimension one (or fractional twist sequence) is a function $\Phi:\mathbb{Z} \to \mathscr{K}$ such that all long virtual knots in the image have identical Gauss diagram presentations outside a proper pair $(A,A')$ and inside $(A,A')$, the diagrams resemble one of the sequences $FYZ$ shown in Figure \ref{regtwist}. The variables $Y$ and $Z$ are as explained above.

A regular twist lattice of dimension $m$ is a function $\Phi:\mathbb{Z}^m \to \mathscr{K}$ together with a $m$ disjoint proper pairs $\{(A_1,A_1'),\ldots,(A_2,A_2')\}$.  Furthermore, it is required that for each of the $m$ canonical inclusions $\iota_i: \mathbb{Z} \to \mathbb{Z}^m $, the composition $\Phi \circ \iota_i: \mathbb{Z} \to \mathscr{K}$ is a regular twist lattice of dimension one, with associated pair of intervals $(A_i,A_i')$, of type $EYZ$ or $OYZ$. For a fractional twist lattice of dimension $m$, it is instead required that $\Phi \circ \iota_i:\mathbb{Z} \to \mathscr{K}$ is of type $FYZ$, $1 \le i \le m$.
\begin{figure}  
\[
\begin{array}{|c|cccc|cccc|} \hline
 & & \multicolumn{3}{c|}{\text{all arrows, } \varepsilon=-} & \multicolumn{3}{c}{\text{all arrows, } \varepsilon=+}& \\ \hline
  & \ldots , & k=-2,  & k=-1, & k=0 & k=1, & k=2, & k=3,& \ldots \\ \hline
OYR & \cdots & \scalebox{.12}{\psfig{figure=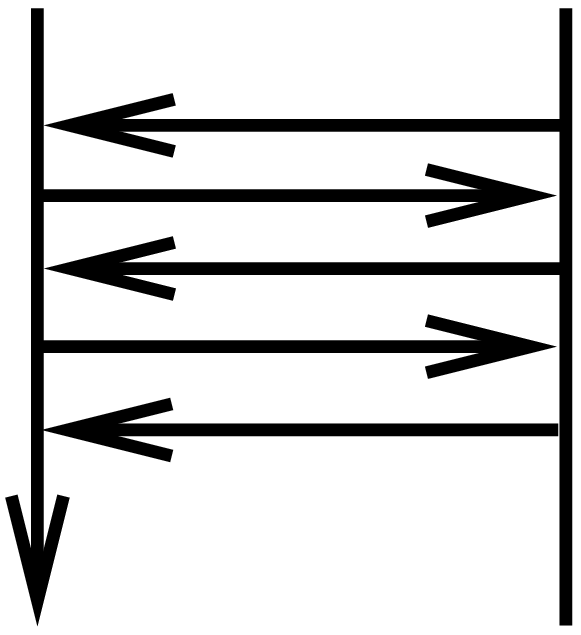}} & 
\scalebox{.12}{\psfig{figure=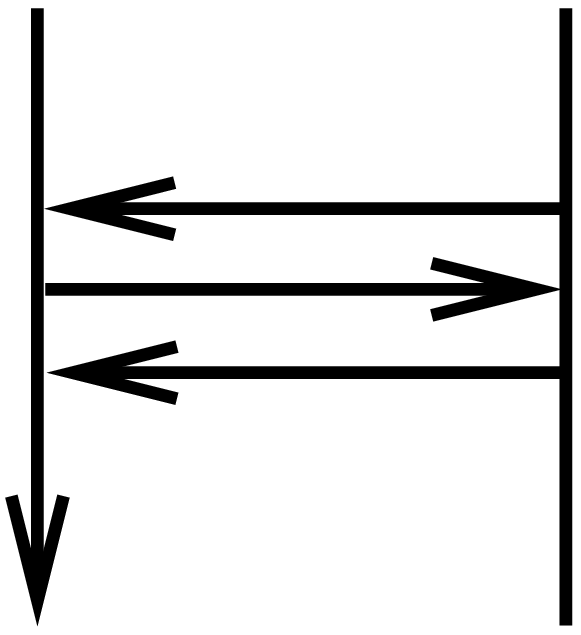}} &
\scalebox{.12}{\psfig{figure=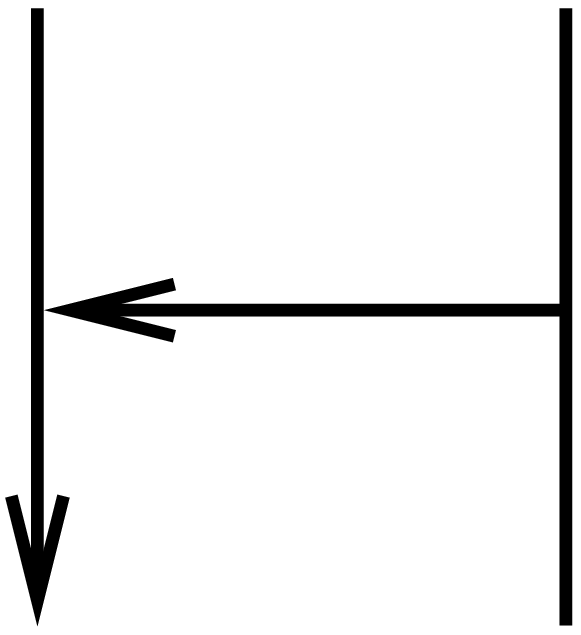}}& 
\scalebox{.12}{\psfig{figure=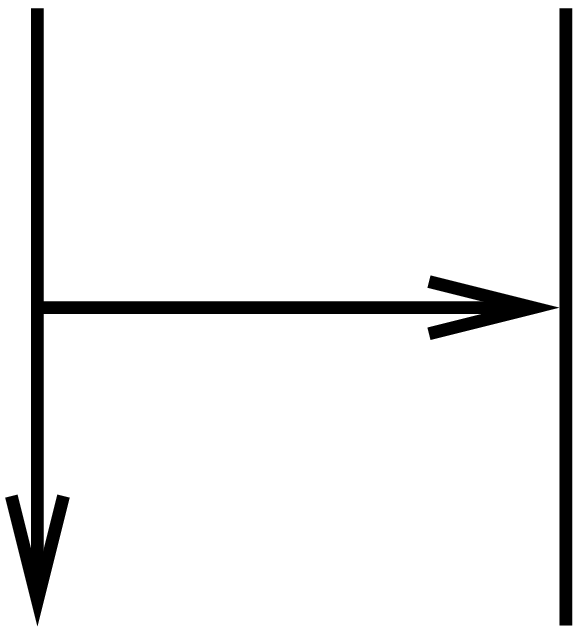}} & 
\scalebox{.12}{\psfig{figure=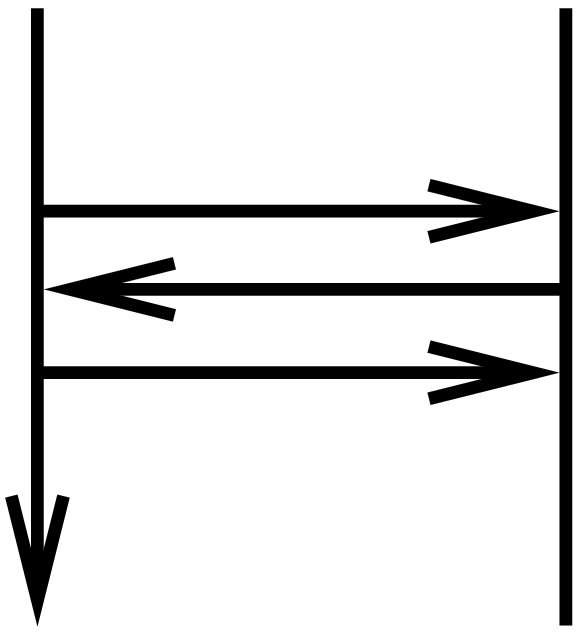}}& 
\scalebox{.12}{\psfig{figure=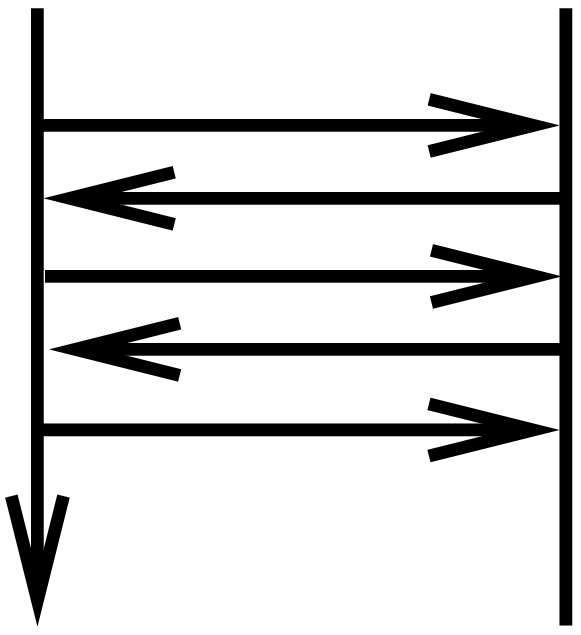}}& \cdots \\ \hline
OYL & \cdots & \scalebox{.12}{\psfig{figure=osrneq3noy.eps}} & 
\scalebox{.12}{\psfig{figure=osrneq2noy.eps}} &
\scalebox{.12}{\psfig{figure=osrneq1noy.eps}}& 
\scalebox{.12}{\psfig{figure=osrneq0noy.eps}} & 
\scalebox{.12}{\psfig{figure=osrneqneg1noy.eps}}& 
\scalebox{.12}{\psfig{figure=osrneqneg2noy.eps}}& \cdots \\ \hline
EYR & \cdots & \scalebox{.12}{\psfig{figure=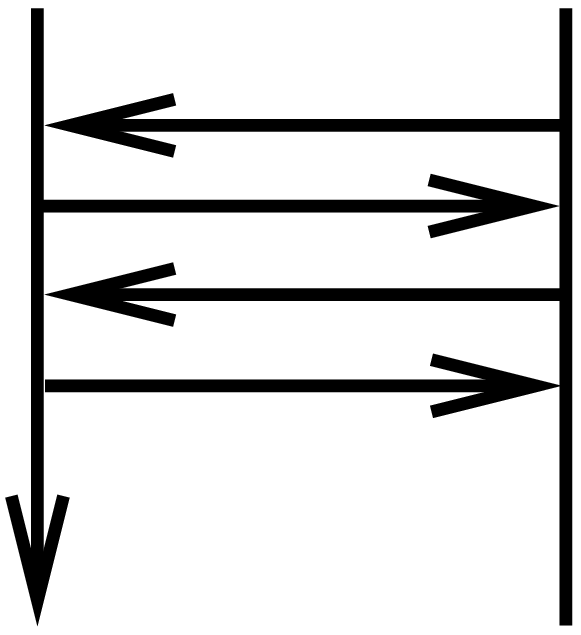}} & 
\scalebox{.12}{\psfig{figure=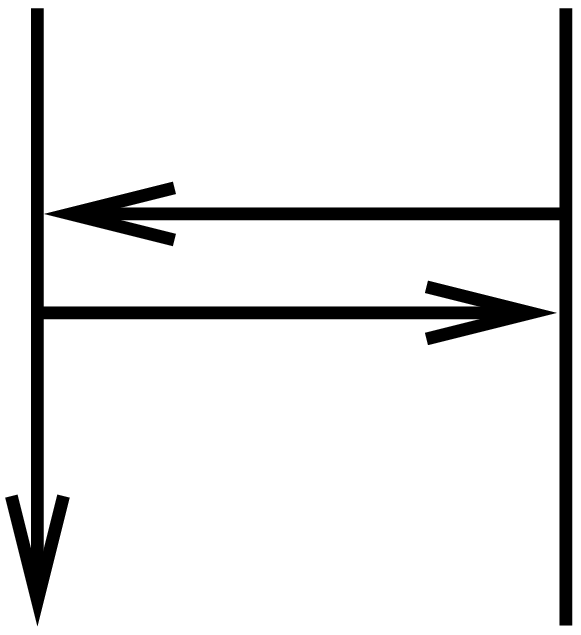}} &
\scalebox{.12}{\psfig{figure=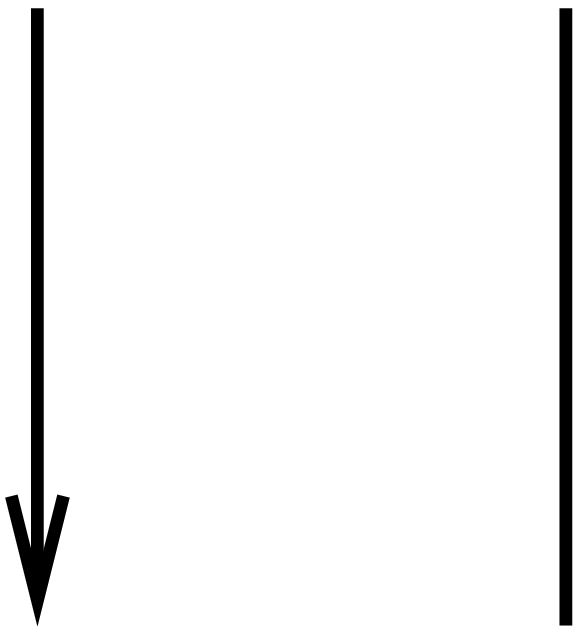}}& 
\scalebox{.12}{\psfig{figure=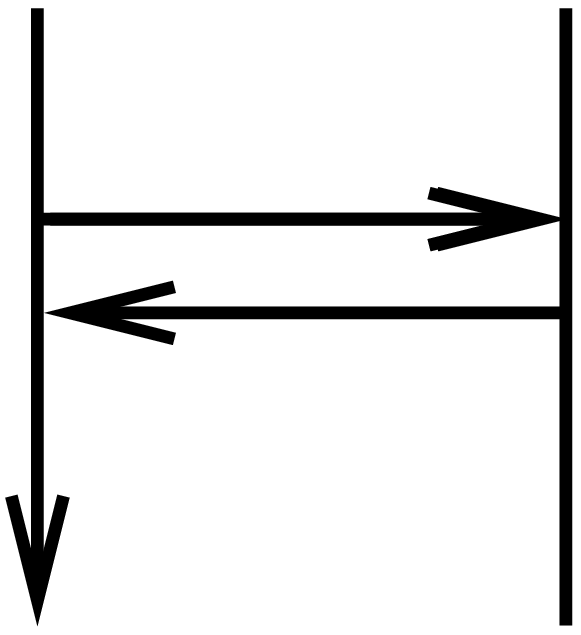}} & 
\scalebox{.12}{\psfig{figure=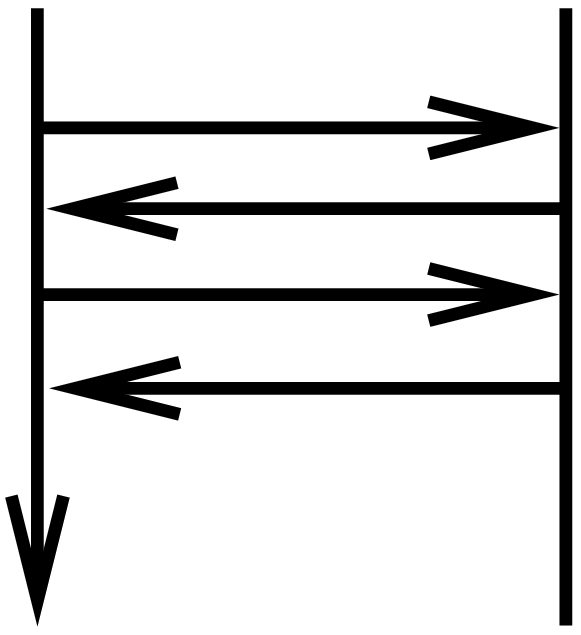}}& 
\scalebox{.12}{\psfig{figure=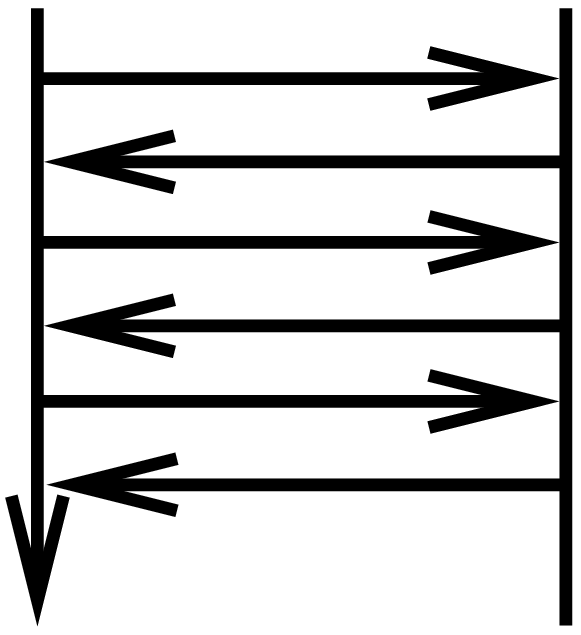}}& \cdots \\ \hline
EYL & \cdots & \scalebox{.12}{\psfig{figure=esrneq2noy.eps}} & 
\scalebox{.12}{\psfig{figure=esrneq1noy.eps}} &
\scalebox{.12}{\psfig{figure=esrneq0noy.eps}}& 
\scalebox{.12}{\psfig{figure=esrneqneg1noy.eps}} & 
\scalebox{.12}{\psfig{figure=esrneqneg2noy.eps}}& 
\scalebox{.12}{\psfig{figure=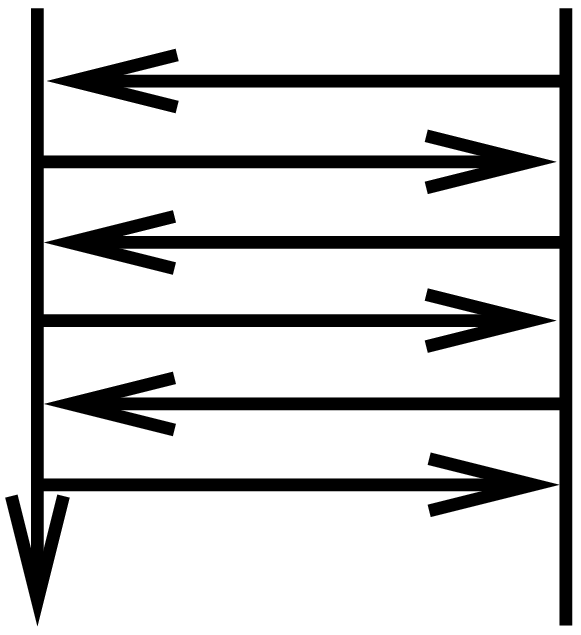}}& \cdots \\ \hline
FYR & \cdots & \scalebox{.12}{\psfig{figure=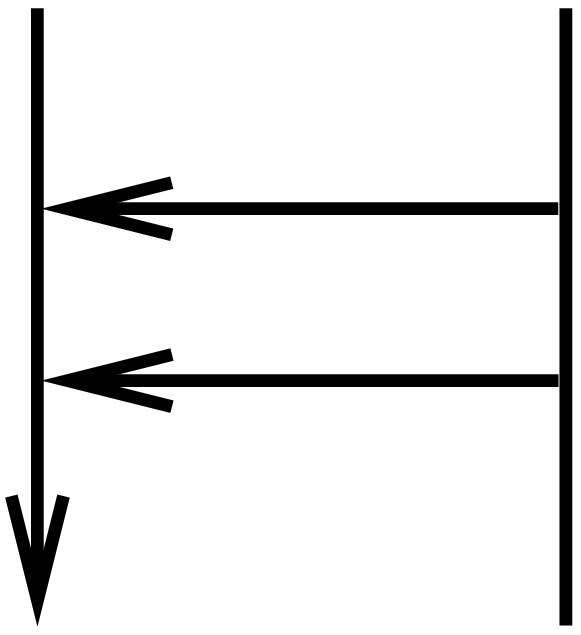}} & 
\scalebox{.12}{\psfig{figure=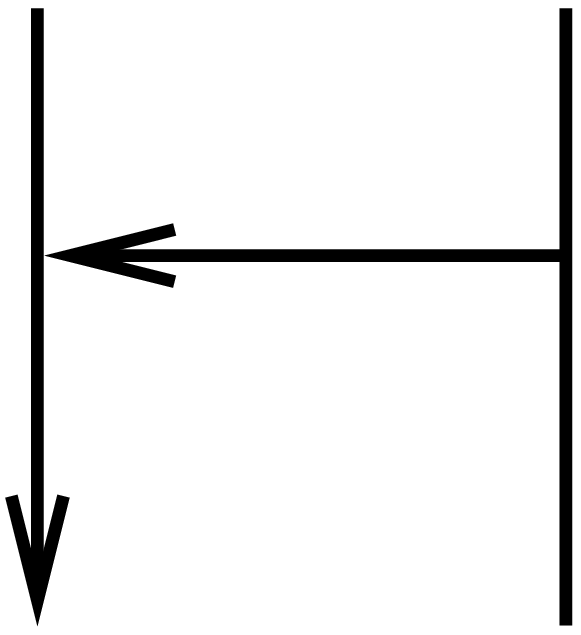}} &
\scalebox{.12}{\psfig{figure=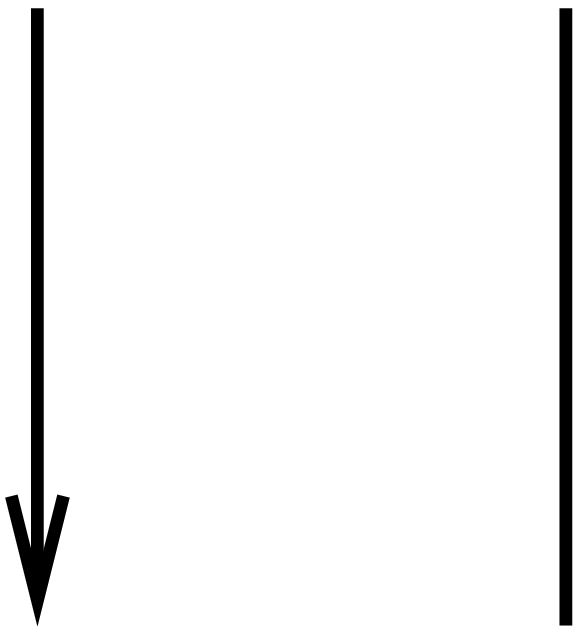}}& 
\scalebox{.12}{\psfig{figure=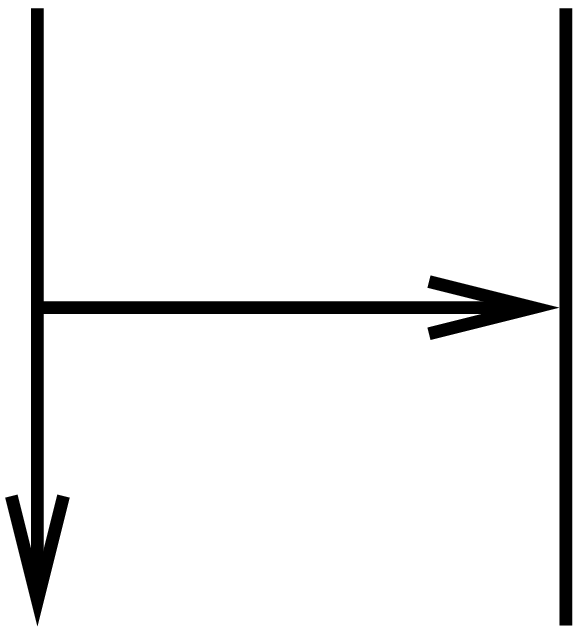}} & 
\scalebox{.12}{\psfig{figure=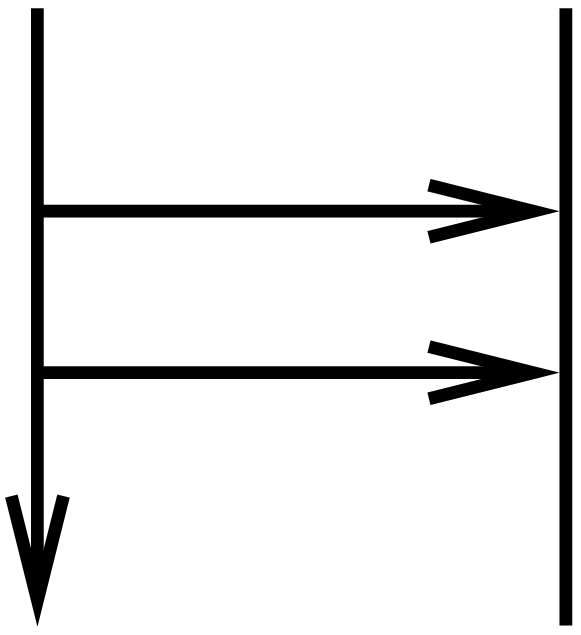}}& 
\scalebox{.12}{\psfig{figure=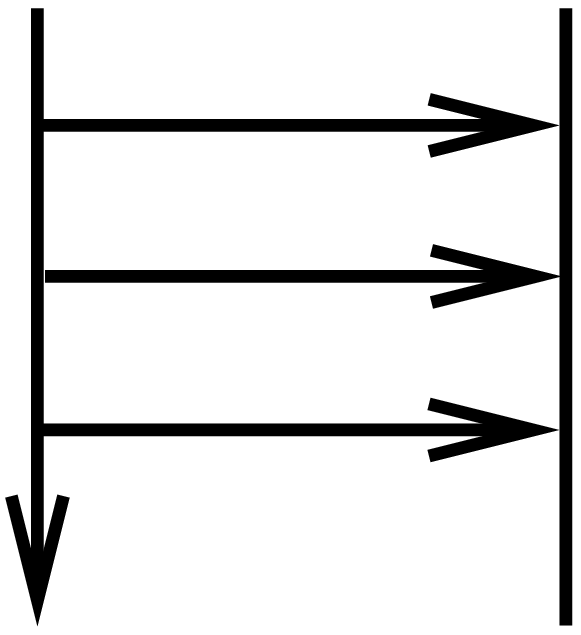}}& \cdots \\ \hline
FYL & \cdots & \scalebox{.12}{\psfig{figure=fsrneq2noy.eps}} & 
\scalebox{.12}{\psfig{figure=fsrneq1noy.eps}} &
\scalebox{.12}{\psfig{figure=fsrneq0noy.eps}}& 
\scalebox{.12}{\psfig{figure=fsrneqneg1noy.eps}} & 
\scalebox{.12}{\psfig{figure=fsrneqneg2noy.eps}}& 
\scalebox{.12}{\psfig{figure=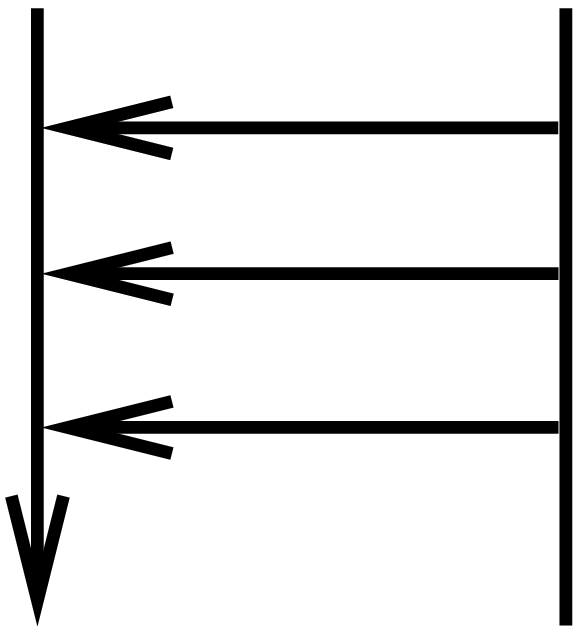}}& \cdots \\ \hline
\end{array}
\]
\caption{Twist Sequences, Type $OYZ$,$EYZ$, $FYZ$} \label{regtwist}
\end{figure}

\subsection{Geometric Characterization} The proof of Theorem \ref{derivthm} requires several lemmas. These lemmas roughly describe the geometric realization of differentiation and antidifferentiation for each of Kauffman and GPV finite-type.

The trick that is used in the proof of the first lemma and the proof of Observation 3 is just the rewriting of the RII move in terms of its Gauss diagram notation.  This is given below for the readers convenience.
\[
\xymatrix{
\scalebox{.15}{\psfig{figure=r2left.eps}} \ar[r]^{\text{RII}} & \ar[l] \scalebox{.15}{\psfig{figure=r2right.eps}} \ar@{<=>}[r] &
\scalebox{.2}{\psfig{figure=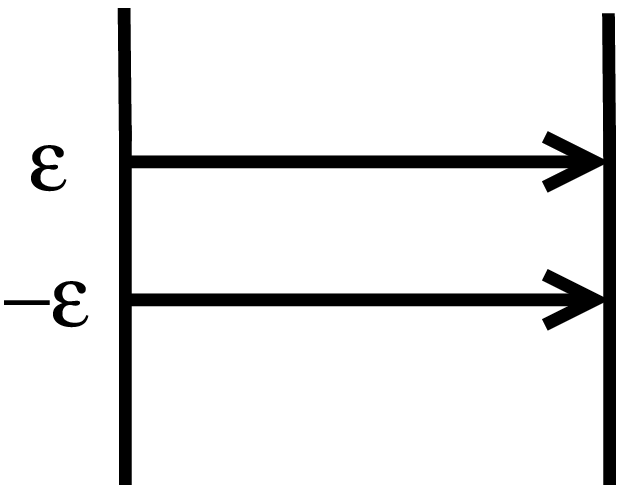}} \ar[r]^{RII} & \ar[l] \scalebox{.2}{\psfig{figure=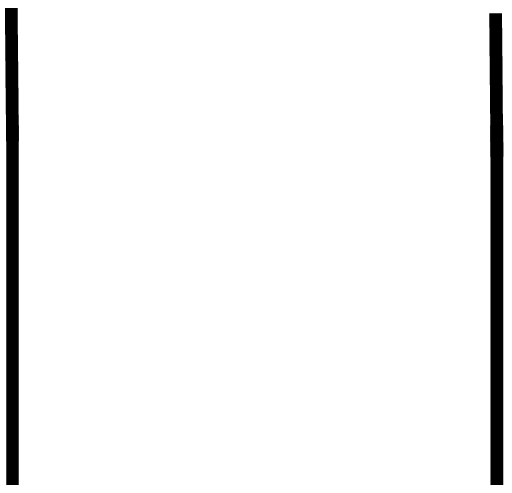}} 
}
\]
\begin{lemma}[Differentiation Lemma I] \label{difflemm1} Let $\alpha \in \mathbb{Z}^+$, $\nu \in \mathbb{Z}_2$, $k \in \mathbb{Z}^{\nu}$. Then for every regular twist sequence $\Phi: \mathbb{Z} \to \mathscr{K}$ of type $XYZ$ and invariant $v:\mathscr{K} \to \mathbb{Q}$, we have $(\partial^{\nu \alpha} v \circ \Phi)(k)=v (K_{\bullet})$, where $K_{\bullet}$ coincides with every element of the image of $\Phi$ outside $(A,A')$ and inside $(A,A')$, we have:
\newline
\[
 K_{\bullet}=\begin{array}{c}
\scalebox{.2}{\psfig{figure=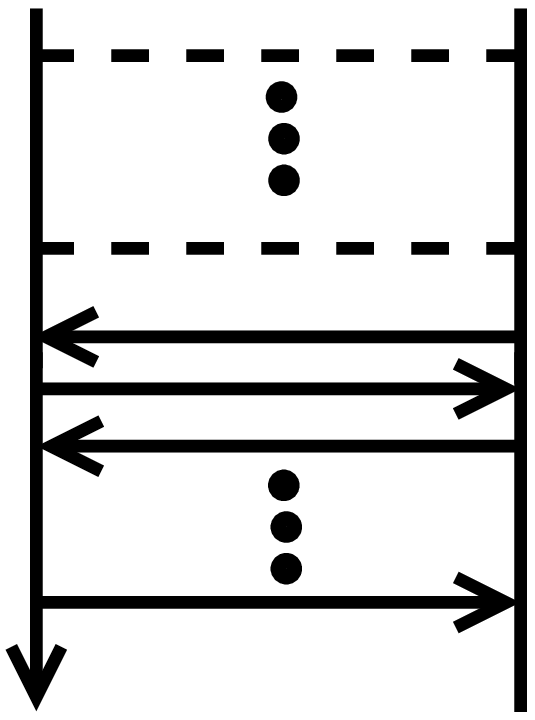}}  
\end{array}
\]
The chords and arrows satisfy the following properties. 
\begin{enumerate}
\item The number of chords in $(A,A')$ is $\alpha$. Moving from left to right in the interval $A$, the sign of the $k-$th chord, $1 \le k \le \alpha$, is $(-1)^{k-1}\varepsilon_X$, where $\varepsilon_L=-$ and $\varepsilon_R=+$.
\item The number of arrows is $\nu\cdot(2k-1)+\alpha$ if $X=O$ and $\nu \cdot 2k+\alpha$ if $X=E$. The first arrow in $(A,A')$ points in the opposite direction of Z if $\alpha$ is odd and in the same direction if $\alpha$ is even.
\end{enumerate} 
\end{lemma}
%There are 7 other similar lemmas, one for each of the remaining types %$XYZ$, $X=O$ or $E$. There is also a lemma for $\alpha<0$.  In general, %we have $(\partial^{\alpha} v \circ \Phi)(k)=v(K_{\bullet})$, where %$K_{\bullet}$ has at least $\alpha$ chords.
\begin{proof} The proof is by induction on $\alpha$. Assume $\nu=+$ and $X=O$. The initial step, $\alpha=1$, and the induction step essentially follow from the same kind of diagram manipulation. We will therefore justify the induction step only.  Suppose the theorem is true up to $\alpha-1$.  Then we have:
\begin{eqnarray*}
(\partial^{\alpha} v \circ \Phi)(k) &=& \partial(\partial^{\alpha-1} v \circ \Phi)(k) \\
                                    &=& (\partial^{\alpha-1}v \circ \Phi)(k+1)-(\partial^{\alpha-1} v \circ \Phi)(k) \\
                                    &=& v \left(\begin{array}{rl} & \alpha-1 \\ & \\ \raisebox{-.1in}[0pt]{\scalebox{.22}{\psfig{figure=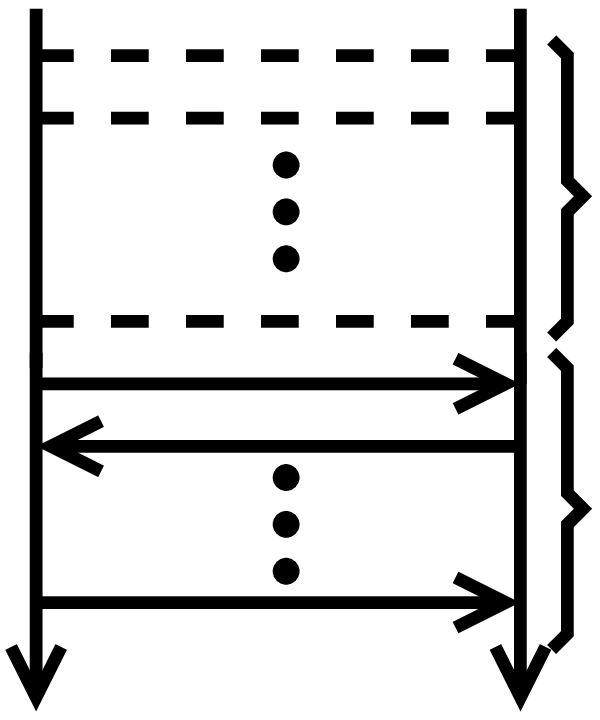}}} &  2k+\alpha \\ \end{array}\right)-v \left(\begin{array}{rl} & \alpha-1 \\ & \\ \raisebox{-.1in}[0pt]{\scalebox{.22}{\psfig{figure=derivfig.eps}}} &  2k-2+\alpha \\ \end{array}\right)\\
&=& v \left(\begin{array}{rl} & \alpha-1 \\ & \\ \raisebox{-.1in}[0pt]{\scalebox{.22}{\psfig{figure=derivfig.eps}}} &  2k+\alpha \\ \end{array}\right)-v \left(\begin{array}{rl} & \\ & \\ \raisebox{-.1in}[0pt]{\scalebox{.22}{\psfig{figure=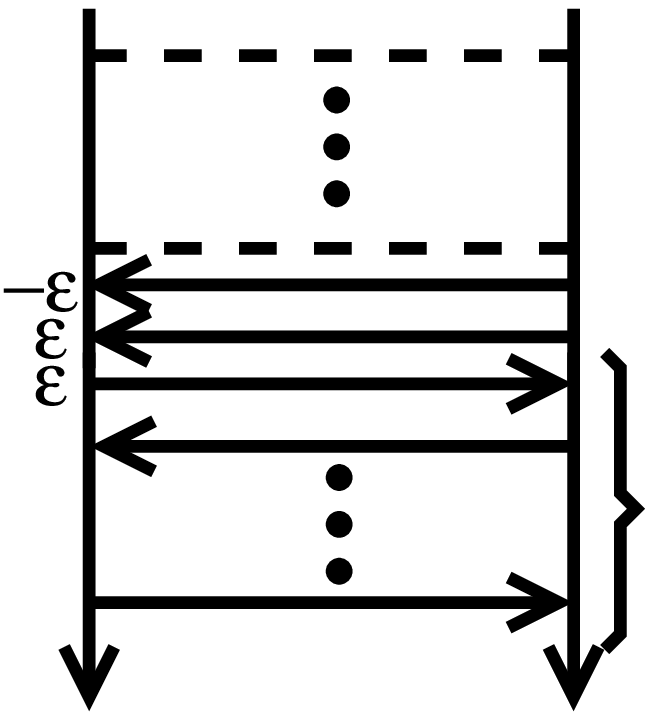}}} &  2k-2+\alpha \\ \end{array}\right)\\
&=& v \left(\begin{array}{rl} & \alpha-1+1 \\ & \\ \raisebox{-.1in}[0pt]{\scalebox{.22}{\psfig{figure=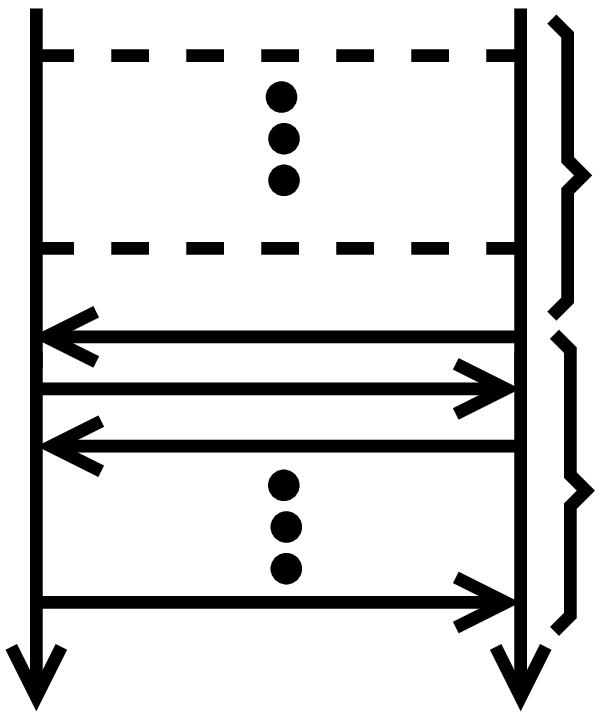}}} &  2k-1+\alpha \\ \end{array}\right)\\                               
\end{eqnarray*}
This follows from the use of an RII move and the definition of a chord.  

From this computation, we see that after a single derivative, the direction of the first arrow changes.  Thus, after an even number of derivatives, the direction will be the same.  An odd number of derivatives always changes the direction.

Now, we see from Table \ref{crossdict} that the sign of a chord is determined by the direction of the arrow signed with a $+$. The computation above shows that the signs of the arrows don't change, but the direction changes.  Hence the sign of each new chord will be different from the previous one.  Table \ref{crossdict} implies that we should set $\varepsilon_L=-$ and $\varepsilon_R=+$. The result follows by induction.  The other cases may be proved similarly.   
\end{proof}
\begin{lemma}[Differentiation Lemma II]Let $\alpha \in \mathbb{Z}^+$, $\nu \in \mathbb{Z}_2$, $k \in \mathbb{Z}^{\nu}$. For every fractional twist sequence $\Phi:\mathbb{Z} \to \mathscr{K}$ of type $FYZ$, we have  $(\partial^{\nu \alpha} v \circ \Phi)(k)=\nu^{\alpha} v(K_{\circ})$, where $K_{\circ}$ is a long virtual knot having $\alpha$ dashed arrows and $k$ regular arrows inside the proper pair of $\Phi$. The arrows and their signs satisfy:
\begin{enumerate}
\item Every dashed or whole arrow is signed $\nu$.
\item Every dashed or whole arrow is directed $Z$.
\end{enumerate}   
\end{lemma}
\begin{proof} Once again, only the induction step in the case $\nu=+$ is verified.  The other cases follow similarly.  For example, consider the sequence of type $FSR$.
\begin{eqnarray*}
(\partial^{\alpha+1} v \circ \Phi)(k) &=& (\partial^{\alpha} v \circ \Phi) (k+1)-(\partial^{\alpha} v \circ \Phi)(k) \\
                                      &=& v \left(\begin{array}{rl} & \alpha \\ & \\ \raisebox{-.25in}[0pt]{\scalebox{.25}{\psfig{figure=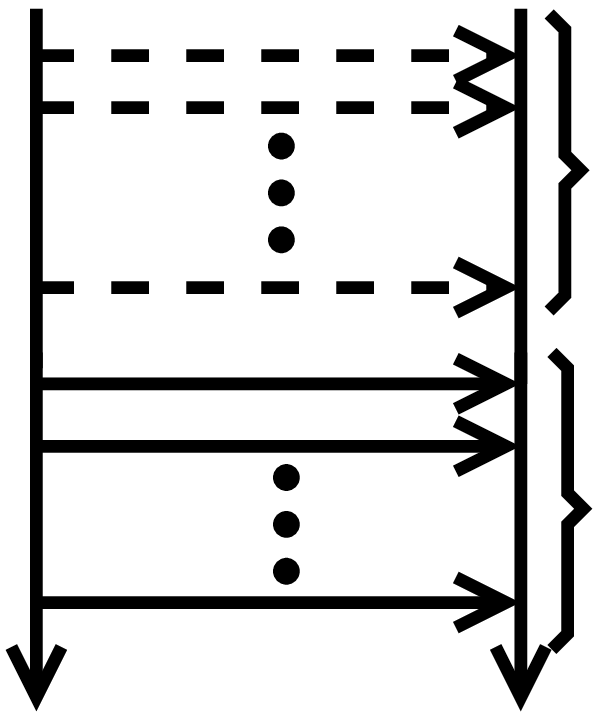}}} & k+1 \\ \end{array}\right)-v \left(\begin{array}{rl} & \alpha \\ & \\ \raisebox{-.25in}[0pt]{\scalebox{.25}{\psfig{figure=fracfig1.eps}}} &  k \\ \end{array}\right)\\
                                      &=&  v \left(\begin{array}{rl} & \alpha+1 \\ & \\ \raisebox{-.25in}[0pt]{\scalebox{.25}{\psfig{figure=fracfig1.eps}}} & k \\ \end{array}\right)
\end{eqnarray*}
The last equality follows from the defining relation for semi-virtual crossings. Note that $\nu^{\alpha}=1$ in this case.

When $\nu=-$, $k\le 0$ and derivatives are of the form:
\begin{eqnarray*}
(\partial^{-\alpha} v \circ \Phi)(k)&=&(\partial^{-\alpha+1}v \circ \Phi)(k)-(\partial^{-\alpha+1} v \circ \Phi)(k-1)\\
&=& -((\partial^{-\alpha+1}v \circ \Phi)(k-1)-(\partial^{-\alpha+1} v \circ \Phi)(k)
\end{eqnarray*}
This case now follows from an induction argument similar to the $\nu=+$ case.
\end{proof}
\begin{lemma} [Integration Lemma I] \label{int1lemm} For every long virtual knot diagram $K$ with $m>0$ chords, there is a regular twist lattice $\Phi_K: \mathbb{Z}^m \to \mathscr{K}$, $\alpha\in (\mathbb{Z}^+)^m$, and $\nu\in\mathbb{Z}_2^m$ such that $\partial^{\nu\alpha} \Phi_K(\vec{0})=K$.
\end{lemma}
\begin{proof} Label the chords $1, \ldots , m$.  Let $\eta_i$ be the sign of the $i$-th chord.  There exists a collection $\{(A_i,A_i'): 1 \le i \le \alpha\}$ of disjoint proper pairs in the Gauss diagram of $K$ such that the $i$-th chord is in the pair $(A_i,A_i')$.  For each pair $(A_i,A_i')$, we create a regular twist lattice of type $OYZ$.  The choice of $Y$ is immaterial. If $\eta_i=+$, choose $Z=R$ and if $\eta_i=-$, choose $Z=L$. This choice of signs is consistent with the choice in Table \ref{crossdict}.  Denote by $\Phi_K:\mathbb{Z}^{m} \to \mathscr{K}$ the regular twist lattice so obtained.  Let $\alpha=(1,\ldots,1)$ and $\nu=(+,\ldots,+)$.  Then by Lemma \ref{difflemm1} and the commutativity of partial derivatives, we have $\partial^{\nu \alpha} \Phi_K(\vec{0})=K$.
\end{proof}
\begin{lemma}[Integration Lemma II] \label{int2lemm} For every dashed arrow diagram $D$ with $m>0$ arrows, there is a fractional twist lattice $\Phi_D:\mathbb{Z}^m \to \mathscr{K}$, $\alpha \in (\mathbb{Z}^+)^m$, $\nu\in \mathbb{Z}_2^m$ such that $I \partial^{\nu \alpha} \Phi_D(\vec{0})=(-1)^{\#(-)} \cdot D$, where $\#(-)$ is the number of arrows of $D$ signed $-$.
\end{lemma}
\begin{proof} Assign to $D$ a fractional twist lattice as follows.  Label the arrows of $D$ as $1,\ldots,m$.  Let $\nu_i$ be the sign of the arrow labeled $i$.  Choose a collection of disjoint proper pairs $\{(A_1,A_1'),\ldots,(A_m, A_m')\}$ such that $(A_i,A_i')$ contains the arrow labeled $i$.  Assign a sequence $FSZ$ to each proper pair according to the following table.
\[
\begin{tabular}{c|c|c|}
$i$ & $L$ & $R$ \\ \hline
$+$ & $L$ & $R$ \\ \hline
$-$ & $R$ & $L$ \\ \hline
\end{tabular}
\]
For example, if the arrow points right and $\nu_i=-$, assign the fractional twist sequence $FSL$.  Performing this for every proper pair gives a fractional twist lattice $\Phi_D:\mathbb{Z}^m \to \mathscr{K}$.  Let $\alpha=\nu \cdot (1, \ldots,1)$. By Lemma \ref{discderiv} and the commutativity of partial derivatives, we have:
\begin{eqnarray*}
\partial^{\nu\alpha} \Phi_D(\vec{0}) &=& \sum_{j_1=0}^1 \cdots \sum_{j_m=0}^1 (-1)^{m+\sum_i (1/2)(\nu_i-1)+j_i} \Phi_D(\nu_1 j_1,\ldots,\nu_n j_n) \\
&=& (-1)^{\sum_i (1/2)(\nu_i-1)} \sum_{j_1=0}^1 \cdots \sum_{j_n=0}^1 (-1)^{m-\sum j_i} \Phi_D(\nu_1 j_1,\ldots,\nu_m j_m)
\end{eqnarray*}
Consider the vectors $( j_1,\ldots, j_m)$.  The collection of these is in one-to-one correspondence with the set $\{0,1\}^m$.  The presence of a 1 corresponds to the presence of an arrow and the presence of a 0 corresponds to the absence of an arrow.  Thus, the sum can be considered over $D' \subset D$, where $D'$ is obtained from $D$ by deleting some number of arrows.  Checking the coefficient, it is seen that $\partial^{\nu\alpha} \Phi_D (\vec{0})=(-1)^{\#(-)} \cdot I^{-1}(D)$, as desired.
\end{proof}
\begin{proof}[Proof of Thm \ref{derivthm}] The proof is essentially the same as Eisermann's result in \cite{MR1997586}.  Suppose $m=1$.  If $v:\mathscr{K} \to \mathbb{Q}$ is a Kauffman finite-type invariant of long virtual knots, then $v(K_{\bullet})=0$ for all $K_{\bullet}$ having more than $n$ double points. For a regular twist lattice $\Phi:\mathbb{Z} \to \mathscr{K}$, we have that $(\partial^{\alpha}v \circ \Phi)(0)=v(K_{\bullet})$, where $K_{\bullet}$ has at least $\alpha$ chords.  If $\alpha>n$, we see that $(\partial^{\alpha} v\circ \Phi)(0)=0$.  Thus, the discrete power series of $v \circ \Phi$ vanishes for all but finitely many terms and we conclude that $v \circ \Phi:\mathbb{Z} \to \mathbb{Q}$ is a polynomial.

For $m>1$, the result follows from the equality of mixed partials.  Let $\alpha \in (\mathbb{Z}^+)^m$. We may write $\partial^{\nu \alpha}=\partial^{\nu_1 \alpha_1} \cdots \partial^{\nu_m \alpha_m}$. It follows from Lemma \ref{difflemm1} that $(\partial^{\nu \alpha} v \circ \Phi)(\vec{0})=v(K_{\bullet})$ where $K_{\bullet}$ has at least $\alpha_1+\alpha_2+\ldots+\alpha_m=|\alpha|$ chords. If $|\alpha|>n$, then the derivative must vanish.  Hence, the discrete power series vanishes for all but finitely many terms. 

Conversely,  let $K_{\bullet}$ be a long virtual knot with $m>n$ chords. By Integration Lemma I, there is a  $\Phi:\mathbb{Z}^m \to \mathscr{K}$, $\alpha \in (\mathbb{Z}^+)^m$, $\nu \in \mathbb{Z}_2^m$, such that $\partial^{\nu\alpha} \Phi(\vec{0})=K_{\bullet}$.  Since $|\alpha|=m>n$, we must have $v(K_{\bullet})=0$.

For the second assertion, the $(\Rightarrow)$ direction follows similarly using Differentiation Lemma II.  To establish the $(\Leftarrow)$ direction, let $D$ be a Gauss diagram having $m>n$ dashed arrows (or equivalently, semi-virtual crossings in the knot projection) and no other arrows of any kind.  By Theorems \ref{iisom} and \ref{GPVuni}, it is sufficient to show that $v(D)=0$.  In this case, $I(D)=D$. By Integration Lemma II, there is a fractional twist lattice $\Phi_D:\mathbb{Z}^m\to \mathscr{K}$, $\alpha \in (\mathbb{Z}^+)^m$, $\nu \in \mathbb{Z}_2^m$ such that $I \circ \partial^{\nu \alpha} \Phi_D(\vec{0})=(-1)^{\#(-)}\cdot D$, where $|\alpha|=m$.  Then:
\[
v(D)=v(I^{-1}(D))=(-1)^{\#(-)} vI^{-1}( I \partial^{\nu\alpha}\Phi_D(\vec{0}))=(-1)^{\#(-)}(\partial^{\nu\alpha} v \circ \Phi_D )(\vec{0})
\]  
Since $|\alpha|>n$, $v(D)=0$.  Thus, $v$ is a GPV finite-type invariant of order $\le n$.
\end{proof}
\section{The Jones-Kauffman Polynomial}
\subsection{Axiomatic definition} The definition of the bracket requires a larger class of objects than $\mathscr{K}$.  Define $\mathscr{E}'$ to be the set of isotopy classes of maps satisfying the following:
\begin{itemize}
\item Embeddings of the form $f:\mathbb{R} \cup S^1 \cup \ldots S^1 \to S^3$, for some finite disjoint union of copies of $S^1$.  Note that there is only one copy of $\mathbb{R}$.
\item The restriction to $\mathbb{R}$ agrees with the standard embedding of the $x$-axis in $S^3=\mathbb{R}^3 \cup \{\infty\}$ outside of some compact set.
\end{itemize}
This collection can be virtualized by considering the collection of regular planar projections of elements of $\mathscr{E}'$ that are embellished at every transversal intersection with a choice of an over crossing, an under crossing, or a four valent graphical vertex (marked with a surrounding circle, as before).  Elements of this set are considered up to equivalence by the three local Reidemeister moves and the four local virtual moves. The totality of equivalence classes is denoted by $\mathscr{E}$.  

When orientation is introduced, a copy of $S^1$ may have any orientation but the immersion of $\mathbb{R}$ must be oriented in the direction of the positive $x$-axis.  This is a standard restriction.

The Kauffman bracket, $<\cdot>:\mathscr{E} \to \mathbb{Q}((A))$ is defined on unoriented diagrams by the following axioms.
\begin{enumerate}
\item $\displaystyle{\left< \begin{array}{c} \scalebox{.15}{\psfig{figure=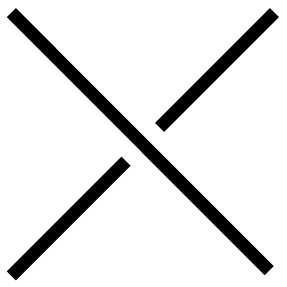}} \end{array} \right>= A \left< \begin{array}{c} \scalebox{.15}{\psfig{figure=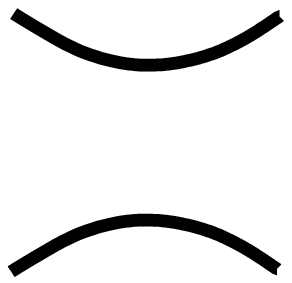}} \end{array}\right>+A^{-1} \left< \begin{array}{c} \scalebox{.15}{\psfig{figure=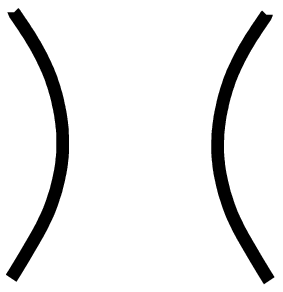}} \end{array}\right>}$
\item $\displaystyle{\left< L \cup \bigcirc \right>}=(-A^2-A^{-2})\left<L\right>$
\item $\displaystyle{\left<\to\right>=1}$
\end{enumerate}
Here, $\bigcirc$ is any closed curve having no classical intersections with itself or with $L$.  The symbol $\to$ corresponds to any long virtual knot having no classical self-intersections. The Jones-Kauffman polynomial, $f_L(A)$, is defined on oriented diagrams using the formula:
\[
f_L(A)=(-A)^{-3 w(L)} \left<\left|L\right|\right>,
\]
where $|L|$ means the virtual long knot $L$ without its orientation.  As usual, the writhe of $L$ is denoted $w(L)$ and is defined to be the sum of the classical crossing signs of $L$. The fact that the axioms define a unique isotopy invariant of classical knots is standard (see \cite{MR1414898}).  The fact that it is a virtual isotopy invariant follows similarly (see \cite{virtkauff} for details in the case of closed virtual knots). In the following, if $L$ is oriented, we will write $<L>\equiv f_L(A)$ as a shorthand.

Besides moves RI-RIII and VrI-VrIV, the Kauffman bracket is invariant with respect to a virtualization move.  Schematically:
\[
\left< \begin{array}{c} \scalebox{.25}{\psfig{figure=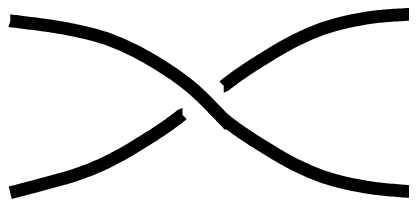}} \end{array} \right>=\left< \begin{array}{c} \scalebox{.25}{\psfig{figure=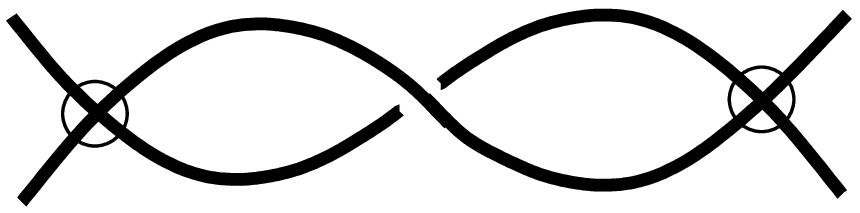}} \end{array}\right>
\]
If one applies any of the possible orientations of the strands, one sees that $w(L)$ is not affected by a virtualization move. Thus $f_L(A)$ is also invariant under the move. In terms of Gauss diagrams, this means that:
\[
\left< \varepsilon \begin{array}{c} \scalebox{.15}{\psfig{figure=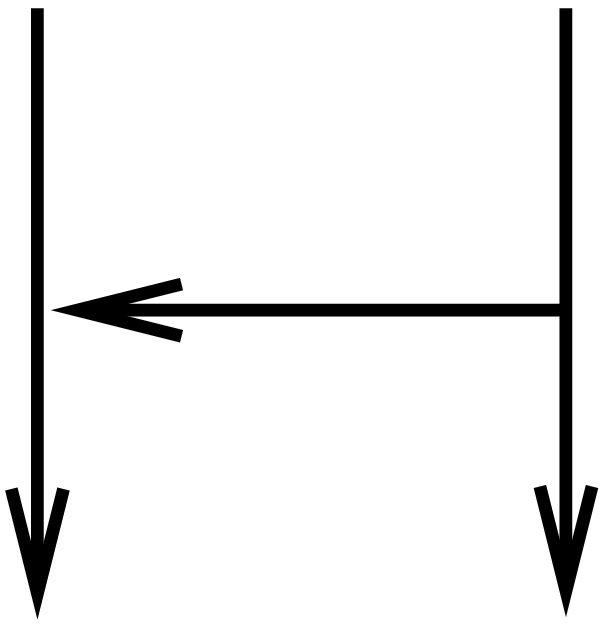}} \end{array}\right>= \left< \varepsilon \begin{array}{c} \scalebox{.15}{\psfig{figure=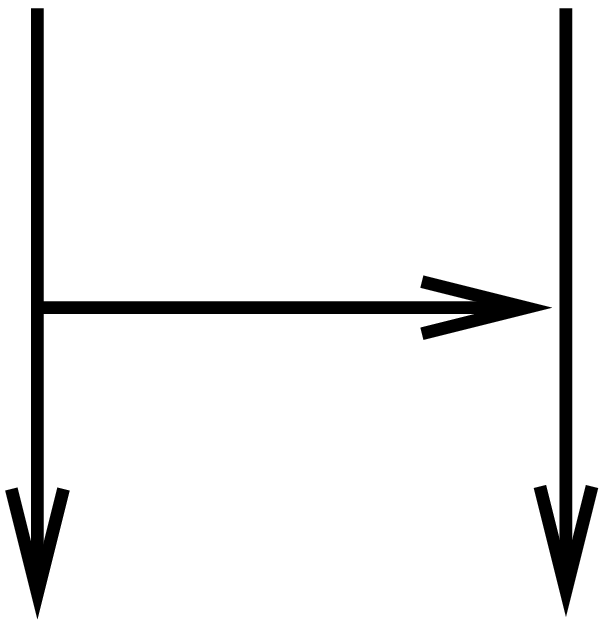}} \end{array}\right>
\]
\begin{observation} \label{obs2} On a proper pair for a fractional or regular twist sequence, any signed arrow may be replaced with an arrow pointing in the opposite direction having the same sign. The Gauss diagram for one such pair resembles the following:
\[
\left< \begin{array}{c} \scalebox{.15}{\psfig{figure=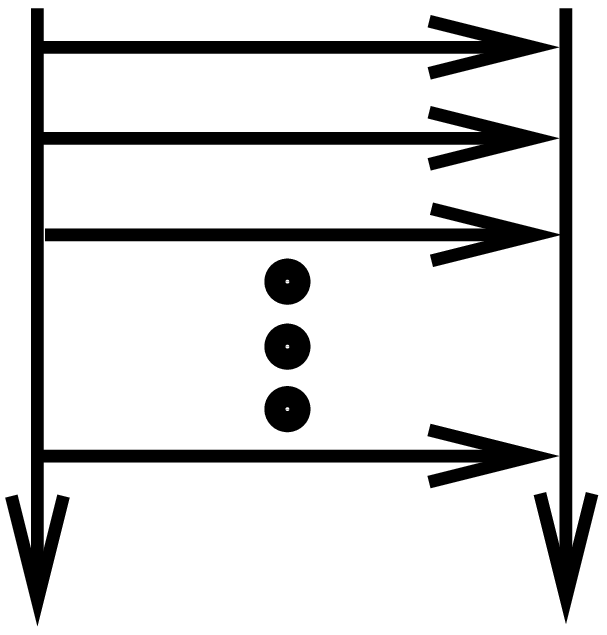}} \end{array}\right>= \left< \begin{array}{c} \scalebox{.15}{\psfig{figure=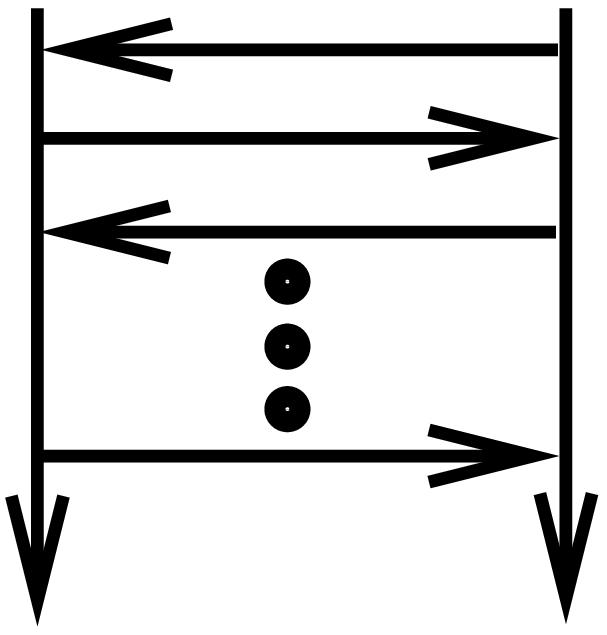}} \end{array}\right>
\]
\end{observation}
\begin{observation} \label{obs3} Using a sequence of virtualizations and RII moves, the Jones-Kauffman polynomial on any fractional twist sequence can be re-indexed so that for some specified $n \in \mathbb{Z}$, $n \to 0$. Below it is shown that after re-indexing, $(n-1)\to -1$:
\[
\begin{array}{c} \left< \begin{array}{c} \scalebox{.2}{\psfig{figure=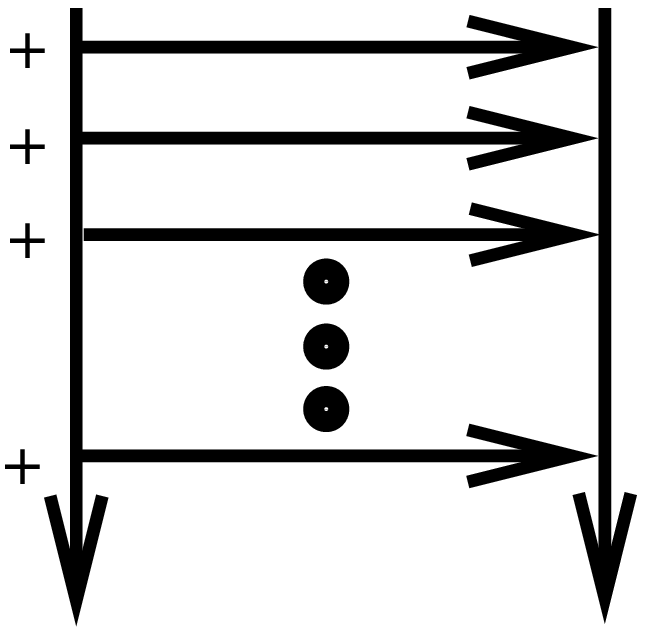}} \end{array} \right> \\ (n-1) \text{ arrows} \end{array}= \begin{array}{c} \left< \begin{array}{c} \scalebox{.2}{\psfig{figure=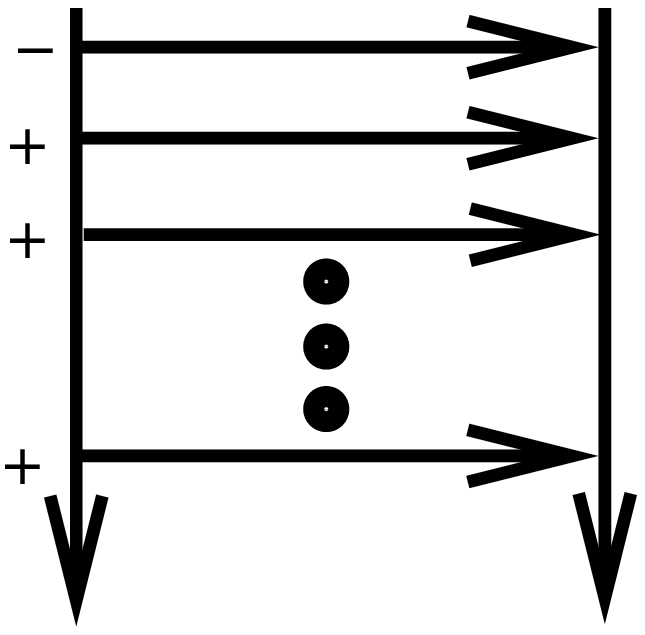}} \end{array}\right> \\ \text{RII move} \end{array}=\begin{array}{c} \left< \begin{array}{c} \scalebox{.2}{\psfig{figure=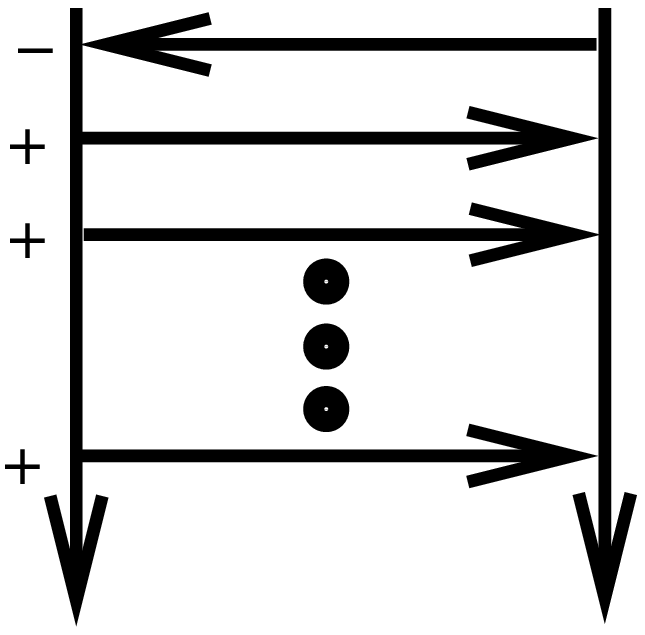}} \end{array} \right> \\ \text{new } \Phi \text{ at } -1 \end{array} 
\]
\end{observation}
\subsection{Kauffman Finite-Type Invariants from the Jones-Kauffman Polynomial} For $K \in \mathscr{K}$, let $v_k(K)$ denote the coefficient of $x^k$ in the power series expansion of $f_K(e^x)$ about $x=0$. Kauffman\cite{virtkauff} has shown that $v_k:\mathscr{K} \to \mathbb{Q}$ is a Kauffman finite-type invariant of degree $\le k$.  In this section, we prove the second main result of this paper: For all $k \ge 2$, $v_k:\mathscr{K} \to \mathbb{Q}$ is not a GPV finite-type invariant of degree $\le n$ for any $n>0$ because there is no $n$ for which it is a polynomial of degree $\le n$ on every fractional twist sequence.

Consider the fractional twist sequence $\Phi:\mathbb{Z} \to \mathscr{K}$ of Figure \ref{testseq} given in Gauss diagram notation.  This can be easily realized by a long virtual knot diagram.  From this diagram, one obtains several formulas for the values of $f_{\Phi(n)}(A)$.
\newline
\begin{figure} 
\[
\begin{array}{ccccccc} \cdots & \scalebox{.2}{\psfig{figure=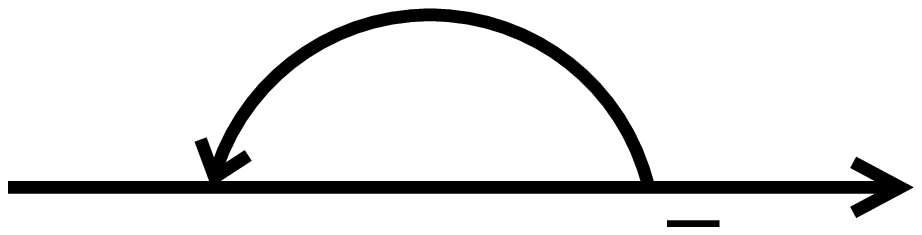}} & \scalebox{.2}{\psfig{figure=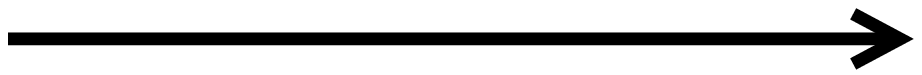}} & \scalebox{.2}{\psfig{figure=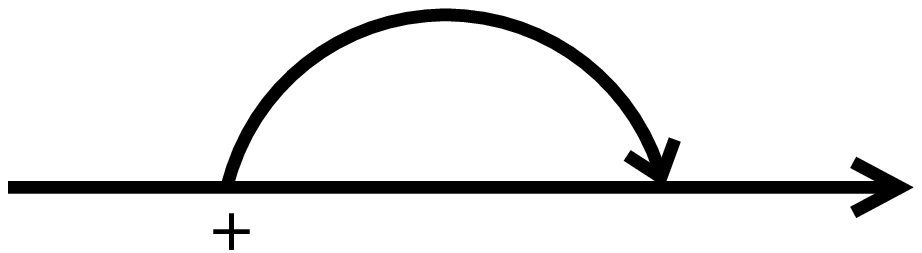}} & \scalebox{.2}{\psfig{figure=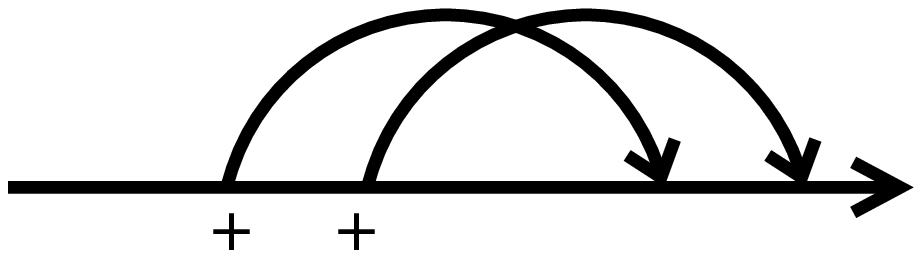}} & \scalebox{.2}{\psfig{figure=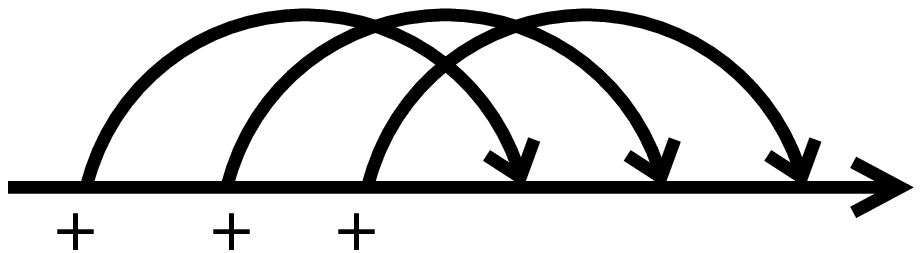}} & \cdots \\
 & n=-1 & n=0 & n=1 & n=2 & n=3 & \\
\end{array}
\]
\centerline{\scalebox{.4}{\psfig{figure=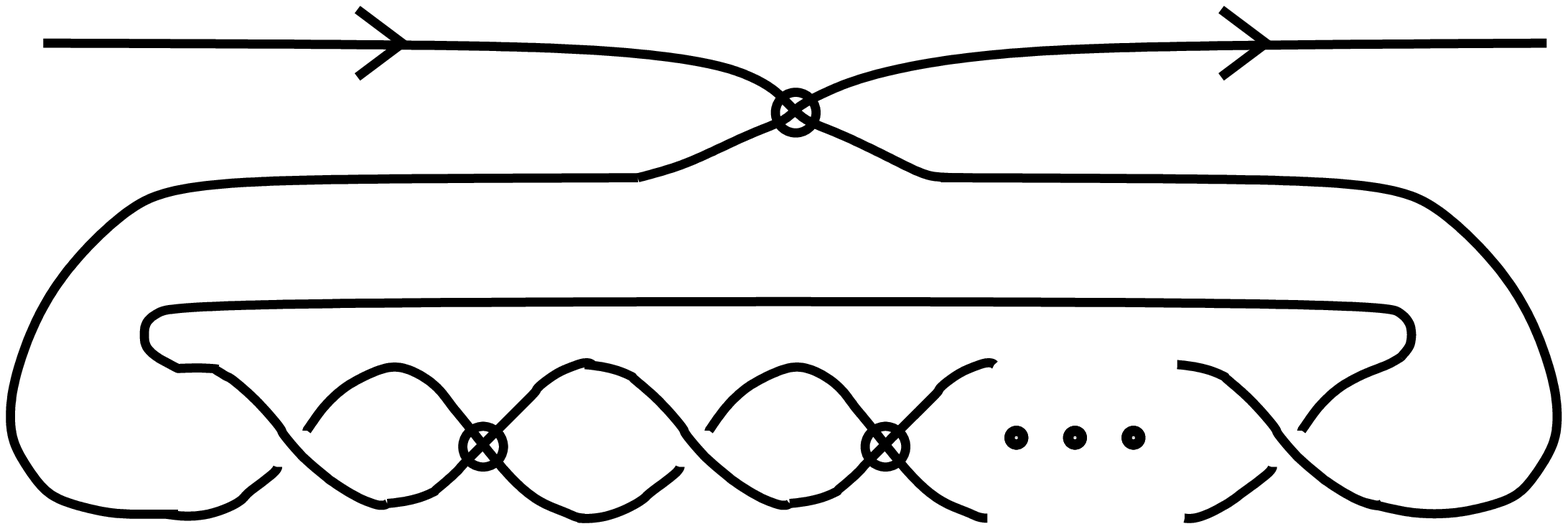}}}
\caption{A Fractional Twist Sequence} \label{testseq}
\end{figure}
\begin{lemma} \label{specpoly} The Jones-Kauffman polynomial for the $n$-th term of this fractional twist sequence is given by:
\begin{enumerate}
\item For $n$ even, $n \ge 0$:
\[
f_{\Phi(n)}(A)=\frac{1}{A^4+1} \left[(A^4+A^2+1)\cdot A^{-2n}-A^{2-6n} \right]
\]
\item For $n$ odd, $n \ge -1$:
\[
f_{\Phi(n)}(A)=\frac{1}{A^4+1} \left[(A^8+A^4+1)\cdot A^{-2-2n}-A^{2-6n}\right]
\]
\end{enumerate}
\end{lemma}
\begin{proof} This follows in the usual way from the axioms for the Jones-Kauffman polynomial and the solution of a recursion equation. This equation follows from the axioms of the Kauffman bracket:
\[
\left< |\Phi(n)| \right>=A^2 \left<|\Phi(n-2)| \right>+(-A)^{-3(n-2)}(1-A^{-4})
\]
The result follows from the data $f_{\Phi(1)}(A)=f_{\Phi(0)}(A)=f_{\Phi(-1)}(A)=1$ and $f_{\Phi(2)}(A)=A^{-4}+A^{-6}-A^{-10}$. 
\end{proof}

\begin{lemma} \label{inductform} Inductive formulas for the coefficients $v_k(\Phi(n))$ are given by:
\[
2 \cdot v_k(\Phi(n))+\sum_{\stackrel{i+j=k}{1 \le i \le k}} \frac{4^i}{i!} \cdot v_j(\Phi(n))= \left\{ \begin{array}{l} \frac{1}{k!}\left[(4-2n)^k+(2-2n)^k+(-2n)^k-(2-6n)^k \right] \\  \\ \frac{1}{k!}\left[(6-2n)^k+(2-2n)^k+(-2-2n)^k-(2-6n)^k \right] \end{array} \right.
\]
The top equation holds when $n \ge 0$ is even; the bottom holds when $n \ge -1$ is odd.
\end{lemma}
\begin{proof} Consider the two formulas in Lemma \ref{specpoly}.  Multiplying by $(A^4+1)$ gives a Laurent polynomial in $A$ of the form $(A^4+1) \cdot f_{\Phi(n)}(A)$ on the left hand side.  Perform the substitution $A \to e^x$ on this form of the equation and expand in a power series about $x=0$.  The two power series on the left hand side can be multiplied together to obtain the desired relationship.  The two possible expressions on the right hand side follow directly from Lemma \ref{specpoly} and the power series about $z=0$ for $e^z$.
\end{proof} 
The main idea of the proof of Theorem \ref{jkthm} is to use the fact that $v_k(\Phi(n))$ is a polynomial of degree $ \le k$ on all regular twist sequences.  By Observation \ref{obs2}, we can write the even terms of a fractional twist sequence as a regular twist sequence of type $ESL$ and the odd terms as a regular twist sequence of type $OSL$.  On these subsequences, $v_k(\Phi(n))$ is a polynomial in the variable $n$ of degree $ \le k$.  The following lemma provides information about the coefficients of these polynomials.
\begin{corollary} \label{coeffdata} For $k \ge 2$, writing $v_k(\Phi(n))$ as a polynomial of degree $\le k$ in the variable $n$ gives the following data on the coefficients:
\begin{enumerate}
\item $\displaystyle{\text{coeff}(n^k)=\frac{(-1)^{k-1}}{2 \cdot k!} (6^k-3\cdot 2^k)}$
\item $\displaystyle{\text{coeff}(n^{k-1})=0}$
\item $\displaystyle{\text{coeff}(n^{k-2})=\left\{  \begin{array}{cl} \frac{(-1)^{k-1}}{(k-2)!}(2^{k-2}-6^{k-2}) & n \text{ is even} \\ \frac{(-1)^{k-1}}{(k-2)!}(-5\cdot 2^{k-2}-6^{k-2}) & n \text{ is odd} \end{array} \right. }$
\end{enumerate}
\end{corollary}
\begin{proof} Consider the inductive formula of Lemma \ref{inductform} leading off with $v_k(\Phi(n))$.  Since $v_j(\Phi(n))$ is a polynomial of degree $\le j$ when $n$ is restricted to be even or $n$ is restricted to be odd, the only $j$ for which the coefficient of $n^k$ is nonzero is $j=k$.  The first formula follows from this observation and the application of the binomial theorem to the right hand side of either equation in Lemma \ref{inductform}.  Note that the same formula holds for $n$ even and $n$ odd.

Apply this result to the coefficient of $n^{k-1}$ in $v_{k-1}(\Phi(n))$. In the inductive formula which leads off with $v_k(\Phi(n))$, the only values of $j$ for which  $v_j$ might have a nonzero coefficient of $n^{k-1}$ are $j=k, k-1$.  Hence, we have the following equation:
\[
2 \cdot\text{coeff}(n^{k-1})+\frac{4(-1)^{k-2}}{2 \cdot (k-1)!}(6^{k-1}-3\cdot 2^{k-1})=\frac{(-1)^{k-1}}{k!} {k \choose k-1}(6\cdot 2^{k-1}-2\cdot 6^{k-1})
\]
From this, we conclude that $\text{coeff}(n^{k-1})=0$.

For the last statement, specialize to the case where $k$ and $n$ are both odd. From the inductive formula leading with $v_k(\Phi(n))$, the binomial theorem gives a right hand side of:
\[
\frac{1}{2!(k-2)!} \left(4\cdot 6^{k-2}-44\cdot 2^{k-2}\right)
\]
The coefficient of $n^{k-2}$ in $v_{k-2}(\Phi(n))$ has already been determined. The coefficient of $n^{k-2}$ in $v_{k-1}(\Phi(n))$ is $0$.  Thus on the left hand side, the only values of $j$ for which $v_j$ has a nonzero coefficient of $n^{k-2}$ are $j=k, k-2$.  All together, this implies the following formula for the coefficient of $n^{k-2}$ in $v_k(\Phi(n))$:
\[
2 \cdot \text{coeff}(n^{k-2})+\frac{4^2}{2!}\cdot\frac{1}{2}\frac{1}{(k-2)!}(6^{k-2}-3\cdot 2^{k-2})=\frac{1}{2!(k-2)!} \left(4\cdot 6^{k-2}-44\cdot 2^{k-2}\right)
\]
Solving for $\text{coeff}(n^{k-2})$ gives the desired result.  The other cases may be proved similarly.
\end{proof}
\begin{proof}[Proof of Theorem \ref{jkthm}] Once again, specialize to the case where $k$ is odd. Here, the fractional twist sequence $\Phi:\mathbb{Z}\to\mathscr{K}$ is defined as above. Define $\Phi^o:\mathbb{Z}\to\mathscr{K}$, $\Phi^e:\mathbb{Z} \to \mathscr{K}$ by $\Phi^o(z)=\Phi(2z-1)$ and $\Phi^e(z)=\Phi(2z)$. On these coordinates, $v_k\circ\Phi^o$ and $v_k\circ\Phi^e$ are both polynomials of degree $\le k$.  This follows from Theorem \ref{derivthm}, Observation \ref{obs2}, and the fact that $v_k$ is a Kauffman type invariant of degree $\le k$.

After a change of variables, $v_k\circ \Phi$ may be considered as a function defined by the polynomial $v_k\circ\Phi^o$ on the odd integers and $v_k\circ\Phi^e$ on the even integers.  It will be shown that $v_k \circ \Phi$ itself is not a polynomial.  To do this, each of the defining polynomials is extended over $\mathbb{R}$.  Define $p(z)$ and $q(z)$ to be the Lagrange interpolating polynomials of degree $k$ for the sets $P$ and $Q$, respectively (see \cite{MR0519124}, pg. 109).
\begin{eqnarray*}
P &=& \{(0,\Phi(0)),(2,\Phi(2)),\ldots,(2k,\Phi(2k))\} \\
Q &=& \{(-1,\Phi(-1)),(1,\Phi(1)),\ldots,(2k-1,\Phi(2k-1))\}
\end{eqnarray*}
So after an appropriate change of variables, $p(z)$ agrees with $\Phi^e(z)$ and $q(z)$ agrees with $\Phi^o(z)$, whenever $\Phi^e(z)$ and $\Phi^o(z)$ are defined.  Note that $p(z)$ and $q(z)$ are defined for all integers, but the coefficients of $z^j$ are determined entirely by the data in $P$ and $Q$. Hence, the coefficients are given as in Corollary \ref{coeffdata}.  The coefficient of $z^k$ in both $p(z)$ and $q(z)$ is the same. Since $k$ is odd, $k\ge 3$, and it follows that $\displaystyle{\lim_{z \to \infty} p(z)=\lim_{z \to \infty} q(z)=\infty}$.  To see which of $p(z)$ and $q(z)$ is eventually larger, it is necessary only to compare the coefficients of $z^{k-2}$:
\begin{eqnarray*}
\text{coeff}(z^{k-2})\text{ in } p(z), k \ge 3 & : & \frac{1}{(k-2)!}(2^{k-2}-6^{k-2})<0 \\
\text{coeff}(z^{k-2})\text{ in } q(z), k \ge 3 & : & \frac{1}{(k-2)!}(-5\cdot 2^{k-2}-6^{k-2})<0 \\ 
\end{eqnarray*}
Clearly then, $\displaystyle{\lim_{z \to \infty}(p(z)-q(z))=\infty}$ and therefore $p(z)>q(z)$ for $z$ sufficiently large.  Suppose that $p(z)>q(z)$ for $z>N$, where $N$ is some sufficiently large even integer.  

By Observation \ref{obs3}, there is a fractional twist sequence $\bar{\Phi}:\mathbb{Z} \to \mathscr{K}$ obtained from $\Phi:\mathbb{Z}\to\mathscr{K}$ by shifting $z=N$ to $z=0$ and satisfying:
\[
v_k\circ\bar{\Phi}(z)=v_k\circ\Phi(z+N)
\]
The other polynomials are also shifted by $N$: $\bar{p}(z)=p(z+N)$, $\bar{q}(z)=q(z+N)$.  Now, consider the average $\bar{m}(z)$ of $\bar{p}(z)$ and $\bar{q}(z)$:
\[
\bar{m}(z)=\frac{\bar{p}(z)+\bar{q}(z)}{2}
\]
This is a polynomial defined on every integer(in fact, all of $\mathbb{R}$) and having degree $\le k$.  Note that for all $z \ge 0$, $\bar{q}(z) < \bar{m}(z) < \bar{p}(z)$.  The proof will be completed by showing at most finitely many discrete derivatives vanish.  Let $\alpha \ge 3$.
\begin{eqnarray*}
(\partial^{\alpha} v_k \circ \bar{\Phi})(0)&=& \sum_{j=0}^{\alpha}(-1)^{\alpha+j} {\alpha \choose j} v_k\circ\bar{\Phi}(j) \\
&=& \sum_{j=0}^{\alpha}(-1)^{\alpha+j} {\alpha \choose j} \left\{ \begin{array}{cl} \bar{p}(j) & \text{if } j \text{ is even} \\
                                                                                    \bar{q}(j) & \text{if } j \text{ is odd} \end{array} \right\} \\
&=& \sum_{j=0}^{\alpha}(-1)^{\alpha+j} {\alpha \choose j} \bar{m}(j)+\sum_{\stackrel{j=0}{2j \le \alpha}} (-1)^{\alpha+2j} {\alpha \choose 2j} [\bar{p}(2j)-\bar{m}(2j)] \\
&+& \sum_{\stackrel{j=0}{2j+1\le \alpha}}(-1)^{\alpha+2j+1} {\alpha \choose 2j+1} [\bar{q}(2j+1)-\bar{m}(2j+1)]\\                 
\end{eqnarray*}
Since $\bar{m}(z)$ is a polynomial of degree $\le k$, Observation \ref{obs1} implies that the first sum vanishes when $\alpha > k$.  If $\alpha$ is odd, the second and third sums are both always negative and hence $(\partial^{\alpha} v_k \circ \bar{\Phi})(0)<0$. If $\alpha$ is even, the second and third sums are both always positive and hence $(\partial^{\alpha} v_k \circ \bar{\Phi})(0)>0$.  Thus for $k$ odd, $(\partial^{\alpha} v_k \circ \bar{\Phi})(0)\ne0$ for all but finitely many $\alpha$ and $v_k\circ\bar{\Phi}$ is not a polynomial of degree $\le n$ for any $n$. The case for $k$ even follows similarly.
\end{proof}
It is interesting to note that it is possible to obtain explicit formulas for $v_k(\Phi(n))$.  From this it follows that $v_k(\Phi(n))$ is an increasing function of $n$ when $k$ is odd and decreasing function of $n$ when $k$ is even.  The interpolating polynomials need not be increasing however, which is why it is necessary to look at their asymptotic behavior.  
\begin{corollary} Explicit formulas for computing $v_k(\Phi(n))$ are given by:
\[
v_k(\Phi(n))= \left\{ \begin{array}{cc} \frac{1}{k!} \left[(-2n)^k+\sum_{j=0}^{n-1}(-1)^{j+1}(4j-6n+2)^k \right] & n \text{ is even}, n \ge 2\\ & \\ \frac{1}{k!} \left[(2-2n)^k+\sum_{j=0}^{n-2}(-1)^{j+1}(4j-6n+2)^k \right] & n \text{ is odd}, n \ge 3 \\ \end{array} \right.
\]
\end{corollary}
\begin{proof} The first formula follows from \ref{specpoly} and some algebra. In this case, we have:
\[
f_{\Phi(n)}(A)=A^{-2n} + \sum_{k=0}^{n-1}(-1)^{k+1} A^{4k-6n+2}
\] 
The second formula follows similarly, where:
\[
f_{\Phi(n)}=A^{2-2n}+\sum_{k=0}^{n-2}(-1)^{k+1}A^{4k-6n+2}
\]
\end{proof}
\begin{corollary} For $k$ odd, $n \ge 0$, $v_{k}(\Phi(n))$ is an increasing nonnegative function of $n$.  For $k$ even, $n \ge 0$, $v_k(\Phi(n))$ is a decreasing nonpositive function of $n$. 
\end{corollary}
\begin{proof} This follows from the previous corollary, the fact that $x^k$ is convex for $x \ge 0$, and the relationship between the secant lines of $x^k$ between successive points in $\mathbb{R}$ (see Royden\cite{MR1013117}, pg. 111 Lemma 16).
\end{proof}
\bibliographystyle{plain}
\bibliography{bib_vass}
\end{document}